\documentclass[11pt]{amsart} 
\usepackage{amsthm,amsbsy,amsfonts,amssymb,amscd}
\usepackage{bbm}
\usepackage[enableskew]{youngtab}
\usepackage[mathscr]{euscript} 

\usepackage{enumerate} 
\usepackage{tikz} 
\usetikzlibrary{calc,intersections,through, backgrounds,arrows,decorations.pathmorphing,backgrounds,positioning,fit,petri} 
\usetikzlibrary{calc}

\usepackage{tikz-cd} 
\usepackage[all]{xy}
\xyoption{knot}
\xyoption{poly}

\usepackage{hyperref}

\usepackage[margin=1.1in]{geometry}

\newcommand{\CCC}{\mathsf{C}}

\newcommand{\Ffun}{\mathscr{F}}
\newcommand{\Kfun}{\mathscr{K}}

\newcommand{\Ab}{\mathbf{A\hspace{-0.5mm}b}}
\newcommand{\lbra}{[\![}
\newcommand{\rbra}{]\!]}

\newcommand{\PL}{\mathcal{P}\hspace{-0.15mm}\mathcal{L}}
\newcommand{\IL}{\mathcal{I}\hspace{-0.15mm}\mathcal{L}}

\newcommand{\ext}{\mathrm{ext}}

\newcommand{\III}{\mathsf{I}}
\newcommand{\JJJ}{\mathsf{J}}
\newcommand{\SSS}{\mathsf{S}}

\newcommand{\cK}{\mathcal{K}}

\newcommand{\htt}{\mathsf{ht}}

\newcommand{\FF}{\mathscr{F}}
\newcommand{\GG}{\mathscr{G}}
\newcommand{\RR}{\mathscr{R}}
\newcommand{\DD}{\mathscr{D}}

\newcommand{\LLL}{\mathsf{L}}

\newcommand{\KKK}{\mathsf{K}}

\newcommand{\cF}{{\mathcal{F}}}

\newcommand{\supp}{\mathrm{supp}}

\newcommand {\supplus}{\mathop{{\subset}\llap{\raise 0.5pt\hbox{\normalfont\small+}\hskip 0.5pt}}}

\newcommand{\rR}{\Phi}

\newcommand{\fHom}{{\mathsf{Hom}}}
\newcommand{\fEnd}{{\mathsf{End}}}
\newcommand{\fnul}{{0}}

\newcommand{\tto}{\twoheadrightarrow}

\newcommand{\cB}{\mathcal B}
\newcommand{\cO}{\mathcal O}
\newcommand{\bO}{\mathbf{O}}

\newcommand{\bC}{\mathbf{C}}

\newcommand{\tr}{{\rm{tr}}}

\newcommand{\im}{{\rm{im}\,}}
\newcommand{\id}{\mathrm{id}}

\newcommand{\ad}{\mathrm{ad}}

\newcommand{\Hom}{\mathrm{Hom}} 
\newcommand{\Ext}{\mathrm{Ext}}

\newcommand{\End}{\mathrm{End}}

\newcommand{\dfs}{{/\kern-2pt/}}

\newcommand{\op}{{\rm op}}
\newcommand{\mk}{\Bbbk}
\newcommand{\cC}{\mathcal{C}}
\newcommand{\cG}{\mathcal{G}}

\newcommand{\cA}{\mathcal{A}}
\newcommand{\Ob}{{\rm{Ob}\,}}

\newcommand{\fn}{\mathfrak{n}}
\newcommand{\fg}{\mathfrak{g}}

\newcommand{\fk}{\mathfrak{k}}
\newcommand{\fp}{\mathfrak{p}}

\newcommand{\fl}{\mathfrak{l}}
\newcommand{\fu}{\mathfrak{u}}

\newcommand{\fh}{\mathfrak{h}}
\newcommand{\fb}{\mathfrak{b}}
\newcommand{\mN}{\mathbb{N}}
\newcommand{\mC}{\mathbb{C}}
\newcommand{\mZ}{\mathbb{Z}}

\numberwithin{equation}{section}

\newtheoremstyle{notes} {} {} {} {} {\bfseries} {.} {.5em} {}

\theoremstyle{plain}

\newtheorem{prop}[subsubsection]{Proposition}

\newtheorem{lemma}[subsubsection]{Lemma}

\newtheorem{cor}[subsubsection]{Corollary}

\newtheorem{thm}[subsubsection]{Theorem}

\newtheorem{que}[subsubsection]{Question}

\theoremstyle{remark}

\newtheorem{rem}[subsubsection]{Remark} 

\newtheorem{ddef}[subsubsection]{Definition} 

\swapnumbers

\usepackage{etoolbox}

\usepackage{ytableau}
\usepackage{youngtab}

\makeatletter
\pretocmd{\appendix}{\addtocontents{toc}{\protect\addvspace{10\p@}}}{}{}
\makeatother

\theoremstyle{remark}

\newtheorem{ex}[subsubsection]{Example}

\newtheoremstyle{construction} {} {} {} {} {\bfseries} { } {0pt} {}

\theoremstyle{construction}

\title[Infinite category O]{On an infinite limit of BGG categories O}

\author{Kevin~Coulembier and Ivan Penkov}

\newcommand{\ind}{{\rm Ind}}
\newcommand{\res}{{\rm Res}}



\begin{document} 
\date{} 
\begin{abstract} We study a version of the BGG category $\cO$ for Dynkin Borel subalgebras of root-reductive Lie algebras {$\mathfrak{g}$}, such as $\mathfrak{gl}(\infty)$. We prove results about extension fullness and compute the higher extensions of simple modules by Verma modules. {In addition, we} show that  {our} category  {$\textbf{O}$} is Ringel self-dual and initiate the study of Koszul duality. An important tool in obtaining these results is an equivalence we establish between appropriate Serre subquotients of  category {$\textbf{O}$} for $\fg$ and category $\cO$ for finite dimensional reductive subalgebras of $\fg$.
	\end{abstract}

\maketitle 



\section*{Introduction}

After about 20 years of study of the representation theory of the three infinite dimensional finitary Lie algebras~$\mathfrak{sl}(\infty),\mathfrak{so}(\infty),\mathfrak{sp}(\infty)$, there still is no standard analogue of the Bernstein-Gelfand-Gelfand category $\mathcal{O}$ for {these} Lie algebras. One reason is that each of these Lie algebras has uncountably many conjugacy classes of Borel subalgebras, so potentially there are many ``categories $\mathcal{O}$". Therefore, one is faced with a selection process trying to sort through various options in constructing interesting and relevant analogues of the BGG category $\mathcal{O}$. Existing results concerning integrable $\fg$-modules, as well as primitive ideals in 
$U(\fg)$, for Lie algebras~$\fg$ as above, motivate the study of interesting analogues of category $\mathcal{O}$ for two types of Borel subalgebras. These are the {\em Dynkin Borel subalgebras}, or Borel subalgebras having ``enough simple roots", and, on the other hand, the {\em ideal Borel subalgebras} defined in \cite{PP1} and \cite{PP2}. These nonintersecting classes of Borel subalgebras are ``responsible" for different classes of representations, and naturally lead to different ``categories $\mathcal{O}$".

The case of Dynkin Borel subalgebras is considered in the recent paper \cite{Nam}(see also \cite{NamT}) where a category $\overline{\mathcal{O}}$ is defined. This category consists of all weight modules with finite dimensional weight spaces which carry a locally finite action of the entire locally nilpotent radical of a fixed Dynkin Borel subalgebra. A direct consequence of the definition of  Dynkin Borel subalgebras is that Verma modules are objects of $\overline{\mathcal{O}}$. Nevertheless $\overline{\mathcal{O}}$ is not a highest weight category due to lack of projective or injective {objects}. Another result is that the subcategory of $\overline{\mathcal{O}}$ consisting of integrable modules (integrable modules are direct limits of finite dimensional modules over finite dimensional subalgebras) is a semisimple category. This makes $\overline{\mathcal{O}}$ somewhat similar to the original BGG category ${\mathcal{O}}$. A concrete motivation to study versions of category $\cO$ for this type of Borel subalgebras comes from the representation theory of finite dimensional Lie superalgebras. Through the concept of ``super-duality'', the category of finite dimensional modules over the general linear superalgebra $\mathfrak{gl}(m|n)$ is related to modules in (parabolic subcategories) of category ${\bar\cO}$ for $\mathfrak{gl}(\infty)$, see e.g~\cite{CLW, CWZ}. Such super-duality {exists} also {in} category $\cO$ for $\mathfrak{gl}(m|n)$ and for Lie superalgebras of other types.

On the other hand, in \cite{PS} an analogue of category ${\mathcal{O}}$ is defined, for an ideal Borel subalgebra. Verma modules are not objects of this category, but its subcategory of integrable modules coincides with the nonsemisimple category of tensor modules studied in [DPS]. This latter category is itself an interesting highest weight category. 

The current paper arose from an attempt to understand better the homological structure of the  category $\overline{\mathcal{O}}$ introduced in \cite{Nam}. It turns out that it is convenient to extend Nampaisarn's category to a category $\bO$ whose objects are weight modules which are locally finite with respect to the locally nilpotent radical of a Dynkin Borel subalgebra, but do not necessarily have finite dimensional weight spaces.

Let's give a brief description of the content of the paper. Sections 1 and 2 are of preliminary nature. Here we recall some general facts about abelian categories and about root-reductive Lie algebras. In particular, we go over the notions of Cartan subalgebras, Borel subalgebras and Weyl groups for root-reductive Lie algebras. In Section 3 we collect some basic facts about Verma modules and dual Verma modules. This section reproves some results of \cite{NamT} and \cite{Nam} and explores some of the peculiarities of Verma modules for Borel algebras which are not of Dynkin type.  

From Section 4 on, we only consider Dynkin Borel subalgebras and introduce the category $\bO$. We demonstrate that  $\bO$ {is a Grothendieck category, and that it} decomposes as the product of indecomposable blocks described by the Weyl group orbits in the dual Cartan subalgebra $\mathfrak{h}^*$. This also reproves Nampaisarn's result about linkage in $\overline{\mathcal{O}}$. Next, we study blocks after truncation to upper finite ideals in $\fh^\ast$. We prove equivalence of the categories of the truncated blocks with categories of modules over certain locally finite dimensional associative algebras. We then show that any truncated category $\bO$ is extension full in $\bO$, and also in the category of weight modules. These results allow to transfer certain homological questions in $\bO$ to categories which have enough projective objects. It is an open question whether the entire category $\bO$ is extension full in the category of weight modules, and whether the category $\overline{\mathcal{O}}$ is extension full in $\bO$.

In Section 5, we prove that the Serre quotient category of two appropriately chosen truncations of $\bO$ is equivalent to  
$\mathcal{O}(\mathfrak{g}_n,\mathfrak{b}_n)$ for arbitrarily large $n$, where $\fg=\varinjlim \fg_n$ for finite dimensional reductive Lie algebras~$\fg_n$. For dominant blocks, it suffices to consider a quotient category of $\bO$, and for antidominant blocks it suffices to consider a subcategory of $\bO$ to establish such equivalences. Using the homological results in Section 4, this shows that the higher extensions of simple modules by Verma modules in $\bO$ are governed by Kazhdan-Lusztig-Vogan polynomials. This was conjectured for $\overline{\cO}$ in \cite{NamT}.
As another application, we show that all blocks of $\bO$ corresponding to integral dominant regular weights are equivalent. In this section we also address the Koszulity of blocks of the category $\bO$. We prove that truncations of $\bO$ admit graded covers, in the sense of \cite{BGS}. In the graded setting, we show that extensions of simple modules by Verma modules (and extensions of dual Verma modules by simple modules) satisfy the Koszulity pattern. For BGG category $\cO$, this property actually implies ordinary Koszulity, see \cite{ADL}. Here, we leave open the question of whether extensions of simple modules by simple modules in the graded cover of $\bO$ also satisfy the required pattern. This is a nice question for further research.

Section 6 and 7 are devoted to another natural structural question: Ringel duality in the category $\bO$. In Section 6 we construct and study the semi-regular $U(\mathfrak{g})$-bimodule. For Kac-Moody algebras corresponding to a finite dimensional Cartan matrix, this bimodule was introduced in \cite{Arkhipov, Soergel}, and we extend the procedure to Kac-Moody algebras for infinite dimensional Cartan matrices, such as $\mathfrak{sl}(\infty),\mathfrak{so}(\infty)$ and $\mathfrak{sp}(\infty)$.
In Section 7 we show that the category $\bO$, as a whole, is Ringel self-dual, by establishing a (covariant) equivalence between the category of modules with a Verma flag and the category of modules with a dual Verma flag. Since this equivalence sends the Verma module $\Delta(\lambda)$ to the dual Verma module $\nabla(-\lambda-2\rho)$, the blocks of $\bO$ are not Ringel self-dual. In particular, dominant blocks are dual to anti-dominant blocks. The Ringel duality functor also implies existence of tilting modules in appropriate Serre quotients and determines their decomposition multiplicities.

The paper is concluded by brief appendices on certain theories to which we refer throughout the text: Serre quotient categories, Ringel duality, graded covers, and quasi-hereditary Koszul algebras.

 {Finally, we would like to mention} some interesting related results, obtained independently at the same time by Chen and Lam in \cite{CL}. There, specific dominant blocks in {\em parabolic subcategories}, with respect to specific Levi subalgebras of finite corank, of $\bar\cO$ for $\mathfrak{gl}(\infty), \mathfrak{so}(\infty), \mathfrak{sp}(\infty)$ are studied. For $\mathfrak{gl}(\infty)$, this leads to categories where the modules have finite length. In that setting, also in \cite{CL} equivalences with the finite rank case are shown and used to obtain results on Koszulity.

\subsection*{Acknowledgement}
KC was supported by ARC grant DE170100623. IP was supported in part by DFG grant PE980/7-1.

\section{Preliminaries}
We fix an algebraically closed field $\mk$ of characteristic zero. For any Lie algebra~$\fk$, the universal enveloping algebra will be denoted by~$U(\fk)$. The restriction functor from the category of $\fk$-modules to the category of $\fl$-modules, for a subalgebra $\fl\subset\fk$, is denoted by $\res^{\fk}_{\fl}$.  We {put} $\mN=\{0,1,2,\ldots\}$. If $A$ is a set{, then} $|A|$ denotes its cardinality. For a category $\cC$ we will usually abbreviate $X\in\Ob\cC$ to $X\in\cC$.

\subsection{Abelian categories}
Let $\cC$ be an arbitrary abelian category.

\subsubsection{Multiplicities}
We follow \cite[Definition~4.1]{Soergel} regarding multiplicities.
For~$M\in \cC$ and simple $L\in\cC$, the multiplicity $[M:L]\in\mN\cup\{\infty\}$ of~$L$ in~$M$ is 
$$[M:L]\;=\;\sup_{F_\bullet}|\{i\,|\,F_iM/F_{i+1}M\cong L\}|,$$
where~$F_\bullet$ ranges over all {\em finite} filtrations $0=F_pM\subset\cdots \subset F_{i+1}M\subset F_iM\subset\cdots\subset F_0M=M$.

\subsubsection{Extensions}
 For each $i\in\mN$, we define the extension functor
$$\Ext_{\cC}^i(-,-)\;:\;\cC^{\op}\times\cC\,\to\, \Ab$$
 as in \cite[III.3]{Ve}, see also \cite[Section~2.1]{CM}. For an abelian subcategory $\iota:\cB\hookrightarrow \cC$, the exact inclusion functor~$\iota$ induces group homomorphisms
\begin{equation}\label{XYi}\iota_{XY}^{i}:\;\Ext^i_{\cB}(X,Y)\;\to\;\Ext^i_{\cC}(X,Y),\quad\mbox{ for all $i\in\mN$ and~$X,Y\in\cB$.}\end{equation}
In general, these are neither epimorphisms nor monomorphisms. When all $\iota_{XY}^i$ are isomorphisms, we say that~$\cB$ is {\bf extension full} in~$\cC$. 

\subsubsection{Coproducts}We denote the coproduct of a family $\{X_\alpha\}$ of objects in~$\cC$, if it exists, by $\bigoplus_\alpha X_\alpha$. By definition, we have an isomorphism
\begin{equation}\label{eqCo}\Hom_{\cC}\left(\bigoplus_\alpha X_\alpha,Y\right)\;\stackrel{\sim}{\to}\; \prod_\alpha\Hom_{\cC}(X_\alpha,Y),\quad f\mapsto (f\circ\iota_\alpha)_\alpha.\end{equation}

The following lemma can be generalised substantially, but it will suffice for our purposes.
\begin{lemma}\label{LemCopr}
If for a family $\{X_\alpha\}_\alpha$ of objects in~$\cC$ and~$Y\in\cC$ we have $\Ext^1_{\cC}(X_\alpha,Y)=0$ for all $\alpha$, then $\Ext^1_{\cC}(\bigoplus_\alpha X_\alpha,Y)=0$.
\end{lemma}
\begin{proof}
Represent an element of $\Ext^1_{\cC}(\bigoplus_\alpha X_\alpha,Y)$ as the upper {row} of the following diagram
$$\xymatrix{0\ar[r]& Y\ar@{=}[d]\ar[r]& M\ar[r]^f&\bigoplus_\alpha X_\alpha\ar[r]&0\\
0\ar[r]& Y\ar[r]& M_\beta\ar[r]^{f_\beta}\ar[u]^{\phi_\beta}& X_\beta\ar[r]\ar[u]^{\iota_\beta}&0.
}$$
With $M_\beta$ the pullback of $X_\beta$ in $M$, we obtain the above commuting diagram with exact rows, for every $\beta$. By assumption, there exist $g_\beta: X_\beta\to M_\beta$ with $f_\beta\circ g_\beta=1_{X_\beta}$. Equation~\eqref{eqCo} yields a morphism $g:\bigoplus_\alpha X_\alpha\to M$ such that~$g\circ\iota_\beta =\phi_\beta\circ g_\beta$. Commutativity of the diagram implies ~$f\circ g\circ\iota_\beta=\iota_\beta$. Since $\beta$ is arbitrary, {the} isomorphism~\eqref{eqCo} implies that $f\circ g$ is the identity of $\bigoplus_\alpha X_\alpha${,} and the extension defined by the upper row of the diagram splits.
\end{proof}

\subsubsection{Serre subcategories}
A non-empty full subcategory~$\cB$ of~$\cC$ is a {\bf Serre subcategory} (``thick subcategory'' in \cite{Gabriel}) if for every short exact sequence in $\cC$
$$0\to Y_1\to X\to Y_2\to 0,$$
we have $X\in \cB$ if and only if~$Y_1,Y_2\in \cB$. The exact inclusion functor~$\imath:\cB\to \cC$ is fully faithful (meaning that all $\iota^0_{XY}$ are isomorphisms) and such that also all $\iota^1_{XY}$ are isomorphisms. In addition, it follows immediately that~$\cB$ is a strictly full (full and replete) subcategory.

\subsection{Locally finite algebras} 

\subsubsection{}A $\mk$-algebra $A$ is {\bf locally unital} if there exists a family of mutually orthogonal idempotents~$\{e_\alpha\,|\, \alpha\in\Lambda\}$ for which we have
$$A\;=\;\bigoplus_{\alpha}e_\alpha A\;=\; \bigoplus_\alpha Ae_\alpha.$$
We denote by $A$-Mod the category of left $A$-modules $M$ which satisfy $M=\bigoplus_\alpha e_\alpha M$, or equivalently $AM=M$.

\subsubsection{} A locally unital algebra $A$ is {\bf locally finite} if {$\dim_{\mk}e_\alpha Ae_\beta<\infty$ for all $\alpha,\beta$}. Fur such an algebra we have the full subcategory $A$-mod of $A$-Mod, of modules $M$ which satisfy $\dim_{\mk}e_\alpha M<\infty$ for all $\alpha$. Clearly the projective modules $Ae_\alpha$ are {objects of} $A$-mod, although $A$-mod will generally not contain {\em enough} projective objects.

\subsection{Partial orders} Fix a partially ordered set $(S,\preceq)$. We will denote the induced partial order on any subset of~$S$ by the same notation $\preceq$.

\subsubsection{}\label{secPaOr}
Any two elements~$\lambda,\mu\in S$ determine an {\bf interval} $\{\nu\in S\,|\,\mu\preceq\nu\preceq \lambda\}$.
 An {\bf ideal} $\KKK$ is a subset of~$S$ with the property that~$\lambda\in \KKK$ and~$\mu\preceq \lambda$ implies~$\mu\in\KKK$. An ideal $\KKK$ is {\bf finitely generated} if there exists a finite subset $E\subset \KKK$ such that each $\mu\in\KKK$ satisfies $\mu\preceq\lambda$ for some $\lambda\in E$. A subset $\JJJ { \subset S}$ is {\bf upper finite}, resp. {\bf lower finite}, if for any $\mu\in\JJJ$ there are only finitely many $\lambda\in \JJJ$ for which $\lambda\succeq\mu$, resp. $\lambda\preceq\mu$. A subset $\III{\subset S}$ is  {\bf complete} if it is a union of intervals, {\it i.e.} if $\lambda,\mu\in \III$ and~$\mu\preceq\nu\preceq \lambda$ implies~$\nu\in\III$. A subset $\CCC$ of~$S$ is a {\bf coideal} if $\lambda\in \CCC$ and~$\mu\succeq \lambda$ imply~$\mu\in\CCC$. Clearly, the intersection of an ideal and a coideal is a complete subset. Furthermore, the coideals in~$S$ are precisely the sets $S\backslash\III$ for ideals $\III$ in~$S$.

\subsubsection{}\label{Complete2}To any complete subset $\III\subset S$, we associate two ideals
$$\overline{\III}=\{\mu\in S\,|\, \mu\preceq \lambda,\;\;\mbox{for some $\lambda\in\III $}\}\quad\mbox{and}\quad \mathring{\III}\;=\;\overline{\III}\;\backslash\;\III.$$
A pair of elements {$\mu, \lambda\in S$} is called {\bf remote} if the interval $\{\nu\in S\,|\,\mu\preceq\nu\preceq \lambda\}$ has infinite cardinality. With this convention, incomparable elements are never remote.
A partial order is {\bf interval finite} if every interval is a finite set, or equivalently if no two elements are remote. For a partial order which is interval finite, all finitely generated ideals are upper finite ideals.


\section{Root-reductive Lie algebras and triangular decompositions}

\subsection{Root-reductive Lie algebras}

\subsubsection{}{Recall that a finite dimensional subalgebra $\fk\subset\fl$ of a finite dimnsional algebra $\fl$ is {\bf reductive in $\fl$} if the adjoint action of $\fk$ on $\fl$ is semisimple.}
\label{LocRed} A Lie algebra~$\fg$ over $\mk$ is {\bf locally reductive} if ${\fg}$ has a collection of {nested} subalgebras~$\{\widetilde{\fg}_n{\subset\widetilde{\fg}_{n+1}}\,|\,n\in\mN\}$
such that
$$\fg\;=\;\varinjlim \widetilde{\fg}_n,$$
where, for each $n\in\mN$, $\widetilde{\fg}_n$ is a {finite dimensional} reductive Lie algebra which is reductive in~$\widetilde{\fg}_{n+1}$.

\subsubsection{}\label{DefrrL}
Consider a locally reductive Lie algebra~$\fg$ as above. If, for each $n\in\mN$, there exists a Cartan subalgebra~$\fh_n\subset\widetilde{\fg}_n$ such that $\fh_n\subset\fh_{n+1}$ and such that each root vector in~$\widetilde{\fg}_n$ is also a root vector in~$\widetilde{\fg}_{n+1}$, the Lie algebra~$\fg$ is called
{\bf root-reductive}.

{If $\fg$ is root reducible}, we {are given an} abelian subalgebra~$\fh=\varinjlim\fh_n$. {Such} subalgebras~$\fh\subset\fg$ are known as {\bf splitting maximal toral subalgebras} of~$\fg$. These subalgebras are also {\bf Cartan subalgebras} of $\fg$, according to the definition and results in \cite[Section~3]{DPS}. We will simply use the term ``Cartan subalgebra'' when referring to splitting maximal toral subalgebras. 
 We also introduce the subalgebras~$$\fg_n:=\widetilde{\fg}_n+\fh\;\subset\;\fg.$$

\begin{lemma} \cite[Theorem~4.1]{DPS},~\cite[Theorem~1]{DP}
For any root-reductive Lie algebra~$\fg$, the derived algebra~$[\fg,\fg]$ is a root-reductive Lie algebra which is a countable direct sum of Lie algebras isomorphic to~$\mathfrak{sl}(\infty)$, $\mathfrak{so}(\infty)$, $\mathfrak{sp}(\infty)$ or finite dimensional simple Lie algebras.
\end{lemma}

\subsubsection{}If $\fg$ is a root-reductive Lie algebra with Cartan subalgebra~$\fh$, {there is} a corresponding decomposition into $\fh$-weight spaces
\begin{equation}\label{rootdec}\fg\;=\;\fh\oplus\bigoplus_{\alpha\in\rR}\fg^{\alpha},\end{equation}
for the set of roots $\rR\subset\fh^\ast$. By construction, we have $\dim_{\mk}\fg^\alpha=1$ for each $\alpha\in \rR$, and~$0\not\in\rR$. We denote the subset of roots belonging to~$\fg_n$ as $\rR_n$, for each $n\in\mN$.

\subsubsection{}\label{SecWeightM}We introduce the category~$\bC(\fg,\fh)$ of~$\fg$-modules which are 
semisimple as~$\fh$-modules. For any $\mu\in\fh^\ast$ and~$M\in\bC(\fg,\fh)$, we denote by~$M_\mu$ the $\mu$-weight space in~$M$. By assumption, we have $M=\bigoplus_\mu M_\mu$ for $M\in \bC(\fg,\fh)$. 
 For any module $M\in\bC(\fg,\fh)$, we consider its support $\supp M\subset \fh^\ast$, which is the set of all weights $\mu$ such that~$M_\mu\not=0$. 
 
The full subcategory of~$\bC(\fg,\fh)$ of modules $M$ which satisfy $\dim_{\mk}M_\mu<\infty$ for all $\mu\in\fh^\ast$, is denoted by~$\cC(\fg,\fh)$. This is clearly a Serre subcategory.
  {There is} the duality $M\mapsto M^{\circledast}$ on $\cC(\fg,\fh)$ which takes $M$ to its $\fh$-finite dual, {\it i.e.} to the maximal $\fh$-semisimple submodule of~$M^\ast=\Hom_{\mk}(M,\mk)$. {There is} also  the duality $M\mapsto M^{\vee}$ which twists the action on $M^{\circledast}$ with the {anti-involution $\tau:\fg\to\fg$} which maps~$\fg^{\alpha}$ to~$\fg^{-\alpha}$ for all $\alpha\in\Phi$, and acts as $-1$ on $\fh^\ast$. In particular, we have
  \begin{equation}
  \label{suppDual}
  \supp M\;=\;\supp M^\vee\quad\mbox{for all $M\in\cC(\fg,\fh).$}
  \end{equation}

\begin{rem}
If we apply the definition of $\bC(\fg,\fh)$ to $\fg_n$, and then only consider modules with support belonging to a fixed coset of $\fh^\ast/\mZ\rR_n$, we automatically get an equivalence with a correspondingly defined category for $\widetilde{\fg}_n$. We will therefore freely use results for the finite dimensional reductive Lie algebra $\widetilde{\fg}_n$, for example related to category $\cO$, when working over $\fg_n$.
\end{rem}

\subsection{Triangular decompositions} Fix a root-reductive Lie algebra~$\fg$ with Cartan subalgebra~$\fh$.

\subsubsection{}\label{DefBorel}Choose a subset $\rR^+\subset\rR$ such that $\rR=\rR^+\amalg\rR^-$, with $\rR^-:=-\rR^+$, and $\alpha+\beta\in\Phi^+$ whenever $\alpha,\beta\in\Phi^+$.
Then {let}
$$ \fn^+:=\bigoplus_{\alpha\in\rR^+}\fg^\alpha,\;\; \fn^-:=\bigoplus_{\alpha\in\rR^-}\fg^{\alpha}.$$
The elements of~$\rR^+$, resp. $\rR^-$, which cannot be written as a sum of two other elements of $\rR^+$, resp. $\Phi^-$, are known as {\bf simple roots}. The positive simple roots constitute the subset $\Sigma \subset\rR^+$. 

The {\bf splitting Borel subalgebras} of~$\fg$ are by definition precisely the subalgebras~$\fb:=\fh\oplus\fn^+$ obtained in the above way. (The decomposition $\fb=\fh\oplus\fn^+$ is a direct sum of vector spaces, not of Lie algebras.)
Note that~$\fb^-=\fh\oplus\fn^-,$ the  Borel subalgebra {opposite} to $\fb$ {and containing $\fh$}, is also a splitting Borel subalgebra. We will simply use the term ``Borel subalgebra'' when referring to splitting Borel subalgebras. A Borel subalgebra for $\fg$ leads to Borel subalgebras for $\fg_n$ and~$\widetilde{\fg}_n$:
$$\fb_n:=\fg_n\cap\fb,\qquad \widetilde{\fb}_n:=\widetilde{\fg}_n\cap\fb.$$

\subsubsection{}
For each $\lambda\in\fh^\ast$, we have the corresponding {\bf Verma module} 
\begin{equation}\label{Verma}\Delta^{\fb}_{\fg}(\lambda):=U(\fg)\otimes_{U(\fb)}\mk_\lambda,\end{equation}
where $\mk_\lambda$ is the one dimensional $\fh$-module of weight $\lambda$ with trivial $\fn^+$-action. We will {omit} the indices $\fg$ and~$\fb$ when it is clear which algebras are considered.

We denote by~$\Gamma^+$ the subset of~$\fh^\ast$, consisting of~$0$ and finite sums of elements in~$\rR^+$. The partial order $\le$ on $\fh^\ast$ is defined as
$$\mu\le\lambda\;\;\Leftrightarrow\;\; \lambda-\mu\in\Gamma^+\;\;\Leftrightarrow\;\; \Delta(\lambda)_\mu\not=0.$$
We use the notation $\le_n$ for the partial order on $\fh^\ast$ obtained from the above procedure applied to $\rR_n^+=\rR_n\cap\rR^+$.


\subsection{Parabolic subalgebras}

\subsubsection{}
For a fixed Borel subalgebra~$\fb$, any subalgebra~$\fp\subset\fg$ which contains $\fb$ is called a {\bf parabolic} subalgebra. The reductive part~$\fl\subset\fp$ is spanned by~$\fh$ and all root spaces $\fg^\alpha$ such that both $\fg^\alpha$ and~$\fg^{-\alpha}$ are in~$\fp$. We denote by $\Phi(\fl)\subset\Phi$ the set of roots occurring in $\fl$. We have the corresponding parabolic decomposition (of vector spaces)
$$\fg\;=\;\fu^-\oplus\fl\oplus\fu^+,\qquad\fp=\fl\oplus\fu^+.$$

The following lemma is an easy consequence of the definitions.
\begin{lemma}\label{LemComplete}
For any $\lambda\in\fh^\ast$ and reductive part~$\fl\subset\fg$ of a parabolic subalgebra, the subset $\lambda+\mZ\rR(\fl)\subset\fh^\ast$ is complete for {the partial order} $\le$.
\end{lemma}

\subsection{Induction and restriction}\label{IndRes}
Fix a Borel subalgebra $\fb$ and a parabolic subalgebra $\fp\subset\fg$ with reductive part $\fl$. We have the exact functor
$$\ind^{\fg}_{\fl,+}\,:\; \fl\mbox{-Mod}\;\to\;\fg\mbox{-Mod,}$$
which is given by interpreting $\fl$-modules as $\fp=\fl\oplus \fu^+$-modules with trivial $\fu^{+}$-action, followed by ordinary induction $\ind^{\fg}_{\fp}=U(\fg)\otimes_{U(\fp)}-$.
For any $\lambda\in\fh^\ast$, we also consider {the functor}
$$\res^{\fg}_{\fl,\lambda}\,:\; \bC(\fg,\fh)\;\to\; \bC(\fl,\fh)$$
{which we define as} the ordinary restriction functor followed by taking the maximal direct summand with support in~$\lambda+\mZ\rR(\fl)$.


\subsection{Dynkin Borel subalgebras}
Consider a root-reductive Lie algebra~$\fg$ with Cartan subalgebra $\fh$.

\begin{prop}\label{equivCond}
For a Borel subalgebra $\fh\subset\fb\subset\fg$, the following conditions are equivalent:
\begin{enumerate}[(i)]
\item The elements of~$\Gamma^+$ are the finite sums of elements in~$\Sigma $.
\item We can write $\fg=\varinjlim\widetilde{\fg}_n$ as in \ref{LocRed}, with the additional condition that~$\widetilde{\fg}_n+\fb$ is a (parabolic) subalgebra of~$\fg$ for each $n\in\mN$.
\item The partial order $\le$ is interval finite.
\item The Lie algebra~$\fg$ is generated by~$\fh$ and the simple (positive and negative) root spaces.
\item For each $\lambda\in\fh^\ast$, the Verma module~$\Delta(\lambda)$ is locally $U(\fb)$-finite.
\item For each $\lambda\in\fh^\ast$, the Verma module~$\Delta(\lambda)$ has finite dimensional weight spaces.
\end{enumerate}
If one of the conditions is satisfied, $\fb$ is called a {\bf Dynkin Borel} subalgebra.
\end{prop}
\begin{proof}
First we show that (ii) and (iv) are equivalent. Choose finite subsets $\Sigma_n\subset\Sigma $ for $n\in\mN$, such that~$\Sigma =\cup_n\Sigma_n$ and~$\Sigma_n\subset\Sigma_{n+1}$. Choose also nested subalgebras $\{\fh_n\subset\fh_{n+1}\}$ of $\fh$ with $\varinjlim\fh_n=\fh$. Then we let $\widetilde{\fg}_n$ be the subalgebra of $\fg$ generated by $\fh_n$ and the root vectors corresponding to $\Sigma_n\sqcup -\Sigma_n$. If (iv) is satisfied, it is easy to see that~$\{\widetilde{\fg}_n\}$ satisfies all properties in (ii). Now assume that (ii) is satisfied. Since any $X\in \fg$ is contained in $\widetilde{\fg}_n$ for some $n$, and $\widetilde{\fg}_n$ is generated by $\widetilde{\fg}_n\cap\fh$ and the simple root spaces of $\fg$ which belong to $\widetilde{\fg}_n$, it follows that (iv) is satisfied.

That {conditions} (i) and (iv) are equivalent is clear.

Now we show that {conditions} (i) and (iii) are equivalent. If (i) is satisfied, then $\lambda\ge \mu$ implies that~$\lambda-\mu$ is a finite sum of simple roots. It follows that the interval between $\lambda$ and~$\mu$ is finite. On the other hand, if (i) is not satisfied, we have $\beta\in\rR^+$ such that we can consecutively subtract elements of $\Sigma $ and always obtain an element of $\rR^+$. It follows that the interval between $\beta$ and~$0$ has infinite cardinality.

Consider $\gamma\in\Gamma^+$. By the PBW theorem, there are finitely many ways to write $\gamma$ as a sum of elements in~$\rR^+$ with non-negative coefficients if and only if $\dim_{\mk} \Delta(\lambda)_{\lambda-\gamma}<\infty$, and it is clear that the latter condition is independent of $\lambda\in\fh^\ast$. It follows that {conditions} (i) and (vi) are equivalent.

Now assume that {condition} (i) is satisfied. By the above, also {conditions} (iii) and (vi) are satisfied. We thus have
$$\sum_{\mu\ge \lambda-\gamma}\dim_{\mk}\Delta(\lambda)_{\mu}<\infty,$$
for an arbitrary $\lambda\in\fh^\ast$ and~$\gamma\in\Gamma^+$.
It follows that {condition} (v) is also satisfied, so (i) implies (v).

Now assume that (i) is not satisfied. Then there exists $\beta\in\rR^+$ which is not a finite sum of elements of $\Sigma $. It follows from standard $\mathfrak{sl}_2$-arguments that~$X_{-\beta}v$, with $X_{-\beta}\in \fg^{-\beta}$ and~$v$ the highest weight vector of an arbitrary Verma module, generates an infinite dimensional $U(\fb)$-module. Hence (v) implies (i).
\end{proof}


\subsubsection{}Consider again an arbitrary Borel subalgebra~$\fb$. 
Following \cite[Section~6]{NamT}, we define the {\bf $\fb$-finite root-reductive subalgebra} as the subalgebra $\fl_{\fb}$ of~$\fg$ generated by~$\fh$ and all root spaces for simple roots, with respect to $\fb$. Then $\mZ\rR(\fl_{\fb})=\mZ\Sigma $ and $\fl_{\fb}$ is the reductive part of the parabolic subalgebra $\fl_{\fb}+\fb$.

We have~$\fl_{\fb}=\fg$ if and only if $\fb$ is a Dynkin Borel subalgebra. In general, $\fb\cap{\fl_{\fb}}$ is a Dynkin Borel subalgebra of~$\fl_{\fb}$.


\subsection{The Weyl group} In this section, we consider a Dynkin Borel subalgebra $\fb\supset\fh$ of~$\fg$. By~\ref{equivCond}(ii), we can assume that $\fg=\varinjlim\widetilde{\fg}_n$ where each $\widetilde{\fg}_n+\fb$ is a (parabolic) subalgebra. 

\subsubsection{}The Weyl group $W_n:=W(\widetilde{\fg}_n:\fh_n)$ is naturally a subgroup of~$W_{n+1}$. Moreover, by assumption, the simple reflections of~$W_n$ as a Coxeter group are mapped to simple reflections in~$W_{n+1}$. The infinite Coxeter group 
$$W(\fg:\fh)\,=\,W\,:=\,\varinjlim W_n$$ has a natural action on $\fh^\ast$. For any $\alpha\in \Phi^+$, we denote the corresponding reflection by~$r_\alpha\in W$.

\subsubsection{}\label{rhoshift}It can easily be checked, see e.g. \cite[Corollary~1.8]{Nam}, that there exists $\rho\in\fh^\ast$ such that the restriction $\rho|_{\fh_n^\ast}$ is the half sum of $\widetilde{\fb}_n$-positive roots for $\widetilde{\fg}_n$, for every $n\in\mN$. The dot action of~$W$ on $\fh^\ast$ is given by
$$w\cdot\lambda\;=\;w(\lambda+\rho)-\rho.$$

\subsubsection{} For a weight $\lambda\in\fh^\ast$, the {\bf integral Weyl group} $W[\lambda]$ is the subgroup of~$W$ of elements~$w\in W$ for which $w\cdot\lambda-\lambda\in\mZ\rR$. A weight $\lambda$ is {\bf integral} if $W=W[\lambda]$.
A weight $\lambda$ is {\bf dominant} if $w\cdot\lambda\le \lambda$, for all $w\in W[\lambda]$, and {\bf antidominant} if $w\cdot\lambda\ge \lambda$, for all $w\in W[\lambda]$. A weight is {\bf regular} if $w\cdot\lambda\not=\lambda$, for all $w\in W[\lambda]$.
The {\bf orbit} of a weight $\lambda$ is denoted by~$[\![\lambda]\!]=W[\lambda]\cdot\lambda$.


\section{Verma modules}
Consider a root-reductive Lie algebra~$\fg$ with Cartan subalgebra $\fh$ and Borel subalgebra~$\fb\supset\fh$.

\subsection{Simple and (dual) Verma modules}

Recall the {Verma module}
$$\Delta(\lambda)\;=\;U(\fg)\otimes_{U(\fb)}\mk_\lambda$$ 
from equation~\eqref{Verma}. It is easy to see that {$\Delta(\lambda)$} has a unique maximal submodule. The corresponding simple quotient of~$\Delta(\lambda)$ is denoted by~$L(\lambda)$. 
We will typically use the notation $v_\lambda$ for a non-zero element in $1\otimes \mk_\lambda\subset\Delta(\lambda)$.

The following lemma states the universality property of Verma modules.
\begin{lemma}\label{LemDeltaProj}
For $M\in \bC(\fg,\fh)$ with $M_\nu=0$ for all $\nu> \mu$, {there is} an isomorphism
$$\Hom_{\fg}(\Delta(\mu),M)\;\stackrel{\sim}{\to} \; M_\mu,\quad \alpha\mapsto \alpha(v_\mu).$$
Consequently, we have
$\dim \Hom_{\fg}(\Delta(\mu),M)\;=\; [M:L(\mu)].$
\end{lemma}
\begin{proof}
By adjunction, we have
$$\Hom_{\fg}(\Delta(\mu),M)\cong \Hom_{\fh}(\mk_\mu, M^{\fn^+})\cong \Hom_{\fh}(\mk_\mu, M)\cong M_\mu,$$
where the second isomorphism follows from the assumptions on $\supp M $.
\end{proof} 

\begin{cor}\label{CorExt1}
For $M\in \bC(\fg,\fh)$ with $M_\nu=0$ for all $\nu> \mu$, we have
$$\Ext^1_{\bC(\fg,\fh)}(\Delta(\mu),M)=0.$$
\end{cor}
\begin{proof}
It follows from Lemma~\ref{LemDeltaProj} that $\Delta(\mu)$ is projective in the Serre subcategory of $\bC(\fg,\fh)$ of modules $M$ with $M_\nu=0$ for all $\nu> \mu$.
\end{proof}

\subsubsection{} If $\fb$ is a Dynkin Borel subalgebra, then $\Delta(\lambda)\in \cC(\fg,\fh)$ by Proposition~\ref{equivCond}. In that case we introduce the {\bf dual Verma module}
$$\nabla(\lambda)\;:=\;\Delta(\lambda)^{\vee},$$ where $\vee$ is the duality on $\cC(\fg,\fh)$ of~\ref{SecWeightM}.
It follows from equation~\eqref{suppDual} that~$L(\lambda)\cong L(\lambda)^\vee$.

\subsection{Reduction to root-reductive subalgebras}
Consider a parabolic subalgebra $\fp\supset\fb$ with reductive part $\fl$.
\begin{lemma}\label{LemMult}
For $\lambda,\mu\in\fh^\ast$ with $\lambda-\mu\in\mZ\rR(\fl)$, {the following holds}
\begin{enumerate}[(i)]
\item $[\Delta(\lambda):L(\mu)]=[\Delta_{\fl}(\lambda):L_{\fl}(\mu)];$
\item $\Hom_{\fg}(\Delta(\mu),\Delta(\lambda))\;\cong\;\Hom_{\fl}(\Delta_{\fl}(\mu),\Delta_{\fl}(\lambda)).$
\end{enumerate}
\end{lemma}
\begin{proof}
Part (i) follows from the observations
$$\ind^{\fg}_{\fl,+}\Delta_{\fl}(\lambda)\cong\Delta(\lambda),\qquad \res^{\fg}_{\fl,\lambda}\Delta(\lambda)\cong\Delta_{\fl}(\lambda),$$
$$[\ind^{\fg}_{\fl,+}L_{\fl}(\mu):L(\mu)]=1\qquad\mbox{and}\qquad \res^{\fg}_{\fl,\lambda}L(\mu)\cong L_{\fl}(\mu).$$

By adjunction, we have
$$\Hom_{\fg}(\Delta(\mu),\Delta(\lambda))\;\cong\; \Hom_{\fl}(\Delta_{\fl}(\mu),\Delta(\lambda)^{\fu^+})\;\cong\; \Hom_{\fl}(\Delta_{\fl}(\mu),\Delta_{\fl}(\lambda)),$$
where the last isomorphism follows from weight considerations. This proves part (ii).
\end{proof}

\subsection{Verma modules for Dynkin Borel subalgebras}
Assume that~$\fb$ is a Dynkin Borel subalgebra. By~\ref{equivCond}(ii), without loss of generality we may {\em assume that each $\widetilde{\fg}_n+\fb$ is a parabolic subalgebra of $\fg$.}

\begin{thm}\label{ThmMult}
Consider arbitrary $\lambda,\mu\in\fh^\ast$. For any $n\in\mN$ such that~$\lambda-\mu \in\mZ\rR_n$, we have
\begin{enumerate}[(i)]
\item $[\Delta(\lambda):L(\mu)]\;=\;[\Delta_{n}(\lambda):L_{n}(\mu)];$
\item  $\Hom_{\fg}(\Delta(\mu),\Delta(\lambda))\;\cong\;\Hom_{\fg_n}(\Delta_n(\mu),\Delta_n(\lambda)).$
\end{enumerate}
\end{thm}
\begin{proof}
This is a special case of Lemma~\ref{LemMult}.
\end{proof}

\begin{rem}For integral regular weights, Theorem~\ref{ThmMult}(i) was obtained in \cite[Proposition~3.6]{Nam} through different methods. Our results completely determine the decomposition multiplicities of Verma modules for Dynkin Borel subalgebras in terms of the Kazhdan-Lusztig multiplicities for finite dimensional reductive Lie algebras.

Analogues of Theorem~\ref{ThmMult} for parabolic Verma modules, where the reductive subalgebra of the parabolic subalgebra has finite rank, can be proved using the same method. Analogues for specific parabolic subalgebras with reductive subalgebra of finite {\em co}rank have been proved in e.g.~\cite{CLW, CWZ}.
\end{rem}

\subsubsection{} The {\bf Bruhat order} on $\fh^\ast$ is the partial order $\uparrow$ generated by the relation
$$\mu\uparrow \lambda\qquad\mbox{if}\qquad \mu=r_\alpha\cdot\lambda \;\mbox{ for some }\;\alpha\in \Phi^+\quad\mbox{and}\quad \mu\le \lambda.$$

\begin{cor}\label{CorBGGThm}
Consider arbitrary $\lambda,\mu\in\fh^\ast$. For any $n\in\mN$ such that~$\lambda-\mu \in\mZ\rR_n$, we have
\begin{enumerate}[(i)]
\item $[\Delta(\lambda):L(\mu)]\not=0$ if and only if $\mu\uparrow \lambda$;
\item  $\dim\Hom_{\fg}(\Delta(\mu),\Delta(\lambda))=\begin{cases}1&\mbox{if $\mu\uparrow\lambda$}\\
0&\mbox{otherwise.}\end{cases}$
\end{enumerate}
\end{cor}
\begin{proof}
This follows immediately from Theorem~\ref{ThmMult} and the BGG theorem, see~\cite[Theorem~5.1]{Humphreys} and~\cite[Theorem~4.2(b)]{Humphreys}.
\end{proof}


\begin{prop}\label{PropDN}
{Let} $\lambda,\mu\in\fh^\ast$. {Then}
\begin{enumerate}[(i)]
\item $\Ext^1_{\bC(\fg,\fh)}(\Delta(\lambda),\nabla(\mu))=0$;
\item $\dim_{\mk}\Hom_{\fg}(\Delta(\lambda),\nabla(\mu))=\begin{cases}1&\mbox{if $\lambda=\mu$}\\
0&\mbox{if $\lambda\not=\mu$;}
\end{cases}$
\item $\Ext^1_{\bC(\fg,\fh)}(\Delta(\lambda),\Delta(\mu))=0\quad\mbox{unless $\lambda< \mu$.}$
\end{enumerate}
\end{prop}
\begin{proof}
Part (iii) is a special case of Corollary~\ref{CorExt1}. If $\lambda\not< \mu$, part (i) follows also from Corollary~\ref{CorExt1}. If $\lambda<\mu$, part (i) follows from the previous case by applying $\vee$. Similarly, {for}  $\lambda\not>\mu$ part (ii) follows from Lemma~\ref{LemDeltaProj}, and in the remaining cases {one applies}~$\vee$.
\end{proof}

\begin{rem}
Proposition~\ref{PropDN} was first obtained in \cite[Propositions~3.8 and~3.9]{Nam}.
\end{rem}

\begin{lemma}\label{LemRad}
If $[\fg,\fg]$ is infinite dimensional, the radical of $\Delta(0)$ is not finitely generated.
\end{lemma}
\begin{proof}
The radical $M$ of $\Delta(0)$ is the  {kernel} of  {the surjective homomorphism} $\Delta(0)\tto\mC$.  {Assume that  }$M$ is finitely generated. Then there also exist finitely many {\em weight} vectors in $M$ which generate $M$. Since  $M$ is locally $U(\fb)$-finite,   {we conclude via} the PBW theorem  that there are finitely many weights $\{\lambda_1,\lambda_2,\ldots,\lambda_l\}$ in $\supp  {M}$ such that each $\mu\in \supp {M}$ satisfies $\mu\le\lambda_i$ for some $1\le i\le l$. However, for each simple negative root $\alpha$, the only $\lambda\in\supp M$ for which $\alpha\le\lambda$ is $\lambda=\alpha$. If $[\fg,\fg]$ is infinite dimensional,  {there are} infinitely many such $\alpha${,} and we  {have} a contradiction.
\end{proof}

\subsection{Modules with $\Delta$-flag or $\nabla$-flag} 
Assume that~$\fb$ is a Dynkin Borel subalgebra.
\subsubsection{}

Denote by~$\cF^{\Delta}(\fg,\fb)$, resp. $\cF^{\nabla}(\fg,\fb)$, the full subcategory of modules in~$\bC(\fg,\fh)$ which admit a finite ${\Delta}$-flag, resp. ${\nabla}$-flag. By a ${\Delta}$-flag of~$M$ we mean a filtration
\begin{equation}\label{eqDflag}0=F_kM\subset F_{k-1}M\subset\cdots F_1M\subset F_0M=M,\end{equation}
with $F_iM/F_{i+1}M\cong {\Delta}(\mu_i)$ for some $\mu_i\in\fh^\ast$, for each $0\le i<k$. By Proposition~\ref{equivCond}, the categories $\cF^\Delta$ and $\cF^\nabla$ are actually subcategories of $\cC(\fg,\fh)$. The categories $\cF^\Delta$ and $\cF^\nabla$  are not abelian, but we consider them as exact categories, where the short exact sequences are precisely the short exact sequences in $\cC(\fg,\fh)$ for which every term is in $\cF^\Delta$, resp. $\cF^\nabla$.

 For $M\in \cF^{\Delta}$, we denote  by~$(M:{\Delta}(\lambda))$ the number of indices~$i$ for which~$F_iM/F_{i+1}M$ in \eqref{eqDflag} is isomorphic to~${\Delta}(\lambda)$. It is easy to see that $(M:{\Delta}(\lambda))$ is independent of the chosen filtration, for instance by looking at the character of the modules, or from the following lemma.

\begin{lemma}\label{LemDN}
For $M\in \cF^{\Delta}$ and~$\lambda\in\fh^\ast$, we have
$$(M:{\Delta}(\lambda))\;=\;\dim_{\mk}\Hom_{\fg}(M,{\nabla}(\lambda)).$$
\end{lemma}
\begin{proof}
This follows by induction on the length of the filtration, by applying the properties in Proposition~\ref{PropDN}(i) and (ii).
\end{proof}
{Here is an} alternative characterisation of the categories $\cF^{\Delta}$ and~$\cF^{\nabla}$.

\begin{lemma}\label{stupidlemma}
The category~$\cF^{\Delta}$, resp. $\cF^{\nabla}$, is the full subcategory of~$\bC(\fg,\fh)$ consisting of finite direct sums of modules isomorphic to~$U(\fn^-)$ when considered as~$U(\fn^-)$-modules, resp. finite direct sums of modules isomorphic to~$U(\fn^+)^{\circledast}$ when considered as~$U(\fn^+)$-modules.
\end{lemma}
\begin{proof}
It is clear that objects {of}~$\cF^{\Delta}$, resp. $\cF^{\nabla}$, restrict to finite direct sums of modules isomorphic to~$U(\fn^-)$ when considered as an~$U(\fn^-)$-module, resp. finite direct sums of modules isomorphic to~$U(\fn^+)^{\circledast}$ when considered as $U(\fn^+)$-module{s}.

Now consider $M\in \bC(\fg,\fh)$ such that~$\res^{\fg}_{\fn^-}M\cong U(\fn^-)$. Since $M$ must be a weight module, the element $1\in U(\fn^-)$ corresponds to a one dimensional space of weight~$\lambda$ in~$M$ which must be annihilated by~$\fn^+$ and generates~$M$ as an~$\fn$-module. It follows that~$M\cong {\Delta}(\lambda)$.

Now consider $M\in \bC(\fg,\fh)$ such that~$\res^{\fg}_{\fn^-}M\cong U(\fn^-)^{\oplus k}$ for some $k>1$. Since $M$ is a weight module, as an $\fh$-module $M$ is isomorphic to $\oplus_i {\Delta}(\lambda_i)$ for some weights $\lambda_1,\ldots,\lambda_k$. Without loss of generality we assume that $\lambda_1$ is maximal among these weights. This shows that there is an injective $\fg$-module morphism ${\Delta}(\lambda_1)\hookrightarrow M$. We can then proceed by considering $M/{\Delta}(\lambda_1)$.
\end{proof}

The above lemma has the following three immediate consequences.
\begin{cor}\label{CorAdd}
For~$M,M'\in \bC(\fg,\fh)$, we have that~$M\oplus M'$ belongs to~$\cF^{\Delta}$, resp. $\cF^{\nabla}$, if and only if both $M$ and~$M'$ belong to~$\cF^{\Delta}$, resp. $\cF^{\nabla}$. 
\end{cor}

\begin{cor}\label{resolv}
Consider a short exact sequence in~$\bC(\fg,\fh)$
$$0\to A\to B\to C\to 0$$
with $C\in \cF^{\Delta}$. Then we have $A\in \cF^{\Delta}$ if and only if $B\in\cF^{\Delta}$.
\end{cor}

\begin{cor}\label{CorDua}
The duality functor~$\circledast$, resp. $\vee$, on $\cC(\fg,\fh)$ restricts to a contravariant equivalence of exact categories
$$\circledast: \cF^\Delta({\fg,\fb})\;\tilde\to\; \cF^\nabla({\fg,\fb^-}),\quad\mbox{resp.}\quad\; \vee : \cF^{\Delta}(\fg,\fb)\;\tilde\to\; \cF^{\nabla}(\fg,\fb).$$
\end{cor}





\subsection{Verma modules for non-Dynkin Borel subalgebras}
First we determine all morphism spaces between Verma modules in terms of those for Dynkin Borel subalgebras (in Theorem~\ref{ThmMult}).  
\begin{prop}\label{PropHomVerma}
Consider $\lambda,\mu\in\fh^\ast$.
Let $\fl_{\fb}$ be the $\fb$-finite root-reductive subalgebra of~$\fg$.
\begin{enumerate}[(i)]
\item If $\lambda$ and~$\mu$ are not remote, then 
$$\Hom_{\fg}(\Delta(\mu),\Delta(\lambda))\;\cong\;\Hom_{\fl_{\fb}}(\Delta_{\fl_{\fb}}(\mu),\Delta_{\fl_{\fb}}(\lambda)).$$
\item If $\lambda$ and~$\mu$ are remote, then $\Hom_{\fg}(\Delta(\mu),\Delta(\lambda))\;=\;0.$
\end{enumerate}
\end{prop}
\begin{proof}Part (i) is a special case of Lemma~\ref{LemMult}(ii).

Now we prove part (ii). We take a basis of~$\fn^-$ consisting of root vectors. We extend the partial order $\le$ on $\rR^+$ to a total order $\preceq$ such that all roots of finite length are smaller than all roots of infinite length. Then we take a PBW basis of~$U(\fn^-)$, where each basis element is a product of root vectors, and elements of~$\fg^{-\alpha}$ appear to the left of elements of~$\fg^{-\beta}$ if $\alpha\succ \beta$. An arbitrary weight vector of~$\Delta(\lambda)$ is then of the form
$$w=\sum_{i=1}^k u_i\otimes v=\sum_{i=1}^k u_iv_\lambda,$$
{where} $v\in \mk_\lambda$ and each $u_i$ {is} a PBW basis element of~$U(\fn^-)$.

Let $\mu\le \lambda$ be remote from~$\lambda$ and assume that~$w$ is of weight $\mu$.
{Fix} a minimal positive root $\alpha$ of infinite length such that~$\fg^{-\alpha}$ appears in one of the $u_i$. Now, take $\beta\in\Sigma $ such that~$\alpha-\beta\in \rR^+$. We thus have $[\fg^\beta,\fg^{-\alpha}]\not=0$, and for a non-zero $X\in \fg^\beta$ we consider 
$$Xw=\sum_{i=1}^k [X,u_i]\otimes v.$$
Amongst other possible terms, any $[X,u_i]$ such that $\fg^{-\alpha}$ appears in~$u_i$, has a term (in the natural expansion of~$[X,u_i]$ based on the action of~$X$ on each factor in the product $u_i$) with a factor in~$\fg^{-\alpha+\beta}$ which is by construction a PBW basis element. Moreover, this basis element does not appear in other terms of~$Xw$, by minimality of $\alpha$.
It thus follows that~$X\in \fg^\beta\subset\fn^+$ acts non-trivially on $w$. Consequently, there exists no non-zero morphism from~$\Delta(\mu)$ to~$\Delta(\lambda)$
\end{proof}

\begin{rem}
Proposition \ref{PropHomVerma}(i) was first obtained in \cite[Section~6.4]{NamT}.
\end{rem}

\section{The category~$\bO$}\label{Sec4}
For the rest of the paper, fix a root-reductive Lie algebra~$\fg=\varinjlim\widetilde{\fg}_n$ with Dynkin Borel subalgebra~$\fb=\fh\oplus\fn^+$. In particular, $\fg$ can be a finite dimensional reductive Lie algebra. By~\ref{equivCond}(ii), without loss of generality we may {\em assume that each $\widetilde{\fg}_n+\fb$ is a (parabolic) subalgebra.}

\subsection{Definitions}
{Our} main object of study will be the following abelian category.
\begin{ddef}\label{DefO}
The category~$\bO=\bO(\fg,\fb)$ is the full subcategory of~$\bC=\bC(\fg,\fh)$ of modules~$M$ on which $\fb$ acts locally finitely.
\end{ddef}
It is straightforward to see that {the category} $\bO$ is abelian and closed under direct limits. The reason we restrict to Dynkin Borel subalgebras  {is} that we want  {to study a} category   {containing all} the Verma modules, see Proposition~\ref{equivCond}(vi). Our motivation to  {choose precisely the category} $\bO$  {is} that $\bO$ is a Grothendieck category, see Section~\ref{GroCat}.

The simple objects in~$\bO$ are, up to isomorphism,  the simple highest weight modules~$L(\lambda)$ for $\lambda\in\fh^\ast$.
The categories $\cF^\Delta$ and $\cF^\nabla$ are exact subcategories of $\bO$.



\begin{rem}
Let $\cO(\fg,\fb)$ denote the full subcategory of $\bO(\fg,\fb)$ of finitely generated modules.
In case $\fg$ is finite dimensional (so a reductive Lie algebra){,} the universal enveloping algebra $U(\fg)$ is noetherian and the abelian category $\cO(\fg,\fb)$ is the ordinary BGG category  of \cite{BGG, Humphreys}. In this case, the relation between the categories $\cO$ and~$\bO$ is summarised in Proposition~\ref{FDprop} below. In our generality, the category $\cO(\fg,\fb)$ need not be abelian, see Lemma~\ref{LemRad}, and it is natural to study $\bO(\fg,\fb)$.
\end{rem}

\begin{rem}
In \cite{Nam}, the abelian category~$\bar{\cO}(\fg,\fb)$ is studied, which is the full subcategory of~$\bO(\fg,\fb)$ of modules with finite dimensional weight spaces. We thus have Serre subcategories
$$\xymatrix{
\cC(\fg,\fh) \ar@{^{(}->}[r]&\bC(\fg,\fh)\\
\bar{\cO}(\fg,\fb)\ar@{^{(}->}[r]\ar@{^{(}->}[u]& \bO(\fg,\fb).\ar@{^{(}->}[u]
} 
$$
It follows easily from e.g. Proposition~\ref{equivCond}(vi) that $\cO(\fg,\fh)\subset \bar{\cO}(\fg,\fh)$. Hence  the category $\bar{\cO}(\fg,\fh)$ is another natural abelian enlargement of $\cO(\fg,\fh)$.

\end{rem}

From now on we will leave out the references to $\fg,\fb$ and $\fh$ in $\bO(\fg,\fb)$, $\cC(\fg,\fh)$ etc.

\subsubsection{Serre subcategories by truncation}
Let $\KKK$ be any ideal in~$(\fh^\ast,\le)$. The Serre subcategory~${}^{\KKK}\bO$ of~$\bO$ is defined as the full subcategory of modules~$M$ with $\supp M \subset\KKK$. Clearly, we have 
\begin{equation}\label{eqDtrunc}\Delta(\lambda)\in {}^{\KKK}\bO\quad\Leftrightarrow\quad \lambda\in\KKK\quad\Leftrightarrow\quad L(\lambda)\in{}^{\KKK}\bO.\end{equation}
Similarly, ${}^{\KKK}\bar{\cO}$, respectively $\cF^\Delta[\KKK]$, is the subcategory of $\bar{\cO}$, respectively {of} $\cF^\Delta$, of modules with support in~$\KKK$. The condition for $M\in\cF^\Delta$ to be in $\cF^\Delta[\KKK]$ is equivalently characterised as $(M:\Delta(\lambda))=0$ if $\lambda\not\in\KKK$.

\subsubsection{Upper finite ideals}A special role will be played by ideals $\KKK\subset \fh^\ast$ which are upper finite. We denote by $\cK$ the set of upper finite ideals in $(\fh^\ast,\le)$. Then $\cK$ is a directed set for the partial order given by inclusion.

The following lemma is obvious from the fact that simple highest weight modules have finite dimensional weight spaces.
\begin{lemma}\label{LemKmult}
For an upper finite ideal $\KKK\in\cK$ and~$M\in{}^{\KKK}\bO$, we have 
$$M\in{}^{\KKK}\bar{\cO}\quad\;\Leftrightarrow\; \quad [M:L(\mu)]<\infty\mbox{ for all $\mu\in\KKK$.}$$
\end{lemma}

\begin{rem}
When $\fg$ is not finite dimensional, there exist indecomposable modules in $\bO$ for which $\supp M $ is not upper finite. For instance, when $\lambda$ is integral, regular and antidominant we can consider an infinite chain
$$\lambda=\lambda^0\uparrow\lambda^1\uparrow \lambda^2\uparrow\cdots.$$
By Corollary~\ref{CorBGGThm}(ii), we have morphisms $\Delta(\lambda^i)\to\Delta(\lambda^{i+1})$, for all $i\in\mN$. The $\fg$-module $M:=\varinjlim \Delta(\lambda^i)$ belongs to $\bO$. However, since $M$ is not {an object of} $\overline{\cO}$, this does not yet answer the question, raised in \cite[Section~5.3]{NamT}, of whether there exist indecomposable modules in~$\bar{\cO}$ whose support is not upper finite.
\end{rem}



\subsection{Locally projective modules} \label{SecConstProj}
\begin{thm}\label{ThmProj} Take $\KKK\in\cK$.
\begin{enumerate}[(i)]
\item For each $\mu\in\KKK$, there exists a module $P_{\KKK}(\mu)\in{}^{\KKK}\bar{\cO}\subset {}^{\KKK}\bO$ such that:
\begin{enumerate}[(a)]
\item $\dim_{\mk}\Hom_{{}^{\KKK}\bO}(P_{\KKK}(\mu),-)\;=\;[-:L(\mu)].$
\item $P_{\KKK}(\mu)\in \cF^{\Delta}[\KKK]$ {and} 
$$(P_{\KKK}(\mu):\Delta(\nu))=
[\Delta(\nu):L(\mu)]\qquad\mbox{for all $\nu\in \KKK$}.$$
\item $[P_{\KKK}(\mu):L(\nu)]=0$ unless $\nu\in \lbra\mu\rbra$.
\end{enumerate}
\item The category ${}^{\KKK}\bO$ has enough projective objects. Each projective object is a direct sum of modules isomorphic to $P_{\KKK}(\mu)$ with $\mu\in\KKK$.
\end{enumerate}
\end{thm}

We precede the proof with some discussions and a lemma.

\begin{rem}\label{RemBruhat}${}$
\begin{enumerate}[(i)]
\item The existence of projective objects $P_{\KKK}(\mu)$ in ${}^{\KKK}\bar{\cO}$ was first established through different methods in \cite[Section~4]{Nam}.
\item Even though $P_{\KKK}(\mu)\in{}^{\KKK}\bar{\cO}$, {the category ${}^K\bar{\mathcal{O}}$} does generally not have {enough} projective objects. An example is given by considering a regular integral dominant $\lambda\in\fh^\ast$ and $M:=\bigoplus_{\mu\in\lbra\lambda\rbra}L(\mu)$. By Corollary~\ref{CorBGGThm}(i), we have $\dim M_\nu\le \dim\Delta(\lambda)_\nu$ for all $\nu\in\fh^\ast$, so $M\in \bar{\cO}$. On the other hand, by Theorem~\ref{ThmProj}(i)(b), a projective cover of $M$ has infinite dimensional weight spaces. This answers \cite[Open Question~4.15]{Nam} negatively.
\item It will follow {a posteori} that the condition on the ideal $\KKK\subset\fh^\ast$ in Theorem~\ref{ThmProj} can be weakened to demand that it be upper finite with respect to the Bruhat order $\uparrow$.
\item Given an ideal $\KKK\subset\fh^\ast$ and  $\mu\in\KKK$, the existence of a module $P_{\KKK}(\mu)$ as in Theorem~\ref{ThmProj}(i) can be proved under the weaker assumption that $\{\nu\in\KKK\,|\,\nu\ge \mu\}$ is finite (or even {if just}~$\{\nu\in\KKK\,|\, \mu\uparrow \nu\}$ is finite). 
\end{enumerate}
\end{rem}

We follow the approach of \cite[Section~4]{BGG}, see also \cite{CM}. We fix $\KKK\in\cK$ and $\mu\in \KKK$.

\subsubsection{}\label{DefQ} 
We define a $U(\fb)$-module~$V_\mu^{\KKK}$ with $\supp V_\mu^{\KKK}\subset \KKK${, and} with presentation
\begin{equation}\label{presentation}\bigoplus_{\kappa\in S}U(\fb)\otimes_{U(\fh)}\mk_\kappa\to U(\fb)\otimes_{U(\fh)}\mk_\mu\to V_\mu^{\KKK}\to 0 ,\end{equation}
where~$S$ is a multiset of weights in~$\fh^\ast\backslash \KKK$ such that each $\kappa\in\fh^\ast\backslash \KKK$ appears $\dim U(\fb)_{\kappa-\mu}$ times.
Since the set
$\{\nu\in\KKK\,|\,\nu\ge \mu\}$
is finite, $V_\mu^{\KKK}$ is finite dimensional.
We then define
$$Q_{\KKK}(\mu)\;:=\;U(\fg)\otimes_{U(\fb)} V_\mu^{\KKK}.$$
By construction, we have $Q_{\KKK}(\mu)\in\cF^{\Delta}[\KKK]\subset {}^{\KKK}\bO$. The module $Q_{\KKK}(\mu)$ is generated by a vector~$v_\mu$, which we take in the image of~$\mk_\mu$ under the epimorphism in \eqref{presentation}.

\begin{ex}If $\KKK=\{\nu\in\fh^\ast\,|\,\nu\le \mu\}$, we have $V_\mu^{\KKK}=\mk_\mu$ and thus ${Q}_{\KKK}(\mu)\cong \Delta(\mu)$.
\end{ex}


\begin{lemma}\label{Qk}
The module~$Q_{\KKK}(\mu)$ is projective in~${}^{\KKK}\bO${,} and for any $M\in {}^{\KKK}{\bO}$ we have an isomorphism
$$\Hom_{\fg}(Q_{\KKK}(\mu),M)\;\stackrel{\sim}{\to}\;M_{\mu},\qquad \alpha\mapsto \alpha(v_\mu).$$
\end{lemma}
\begin{proof}
Apply the exact induction functor~$U(\fg)\otimes_{U(\fb)}-$ to the exact sequence~\eqref{presentation}, followed by application of the left exact contravariant functor~$\Hom_{\fg}(-,M)$. This yields an exact sequence
$$0\to \Hom_{\fg}(Q_{\KKK}(\mu),M)\to \Hom_{\fh}(\mk_\mu,M)\to \prod_{\kappa\in S}\Hom_{\fh}(\mk_\kappa,M),$$
where we have used adjunction and equation~\eqref{eqCo}.
The right term is zero since $\supp M \subset \KKK$, which concludes the proof.
\end{proof}

\begin{proof}[Proof of Theorem~\ref{ThmProj}]

First assume that there are no simple modules~$L$ in~${}^{\KKK}\bO$, other than $L{\simeq}L(\mu)$, for which~$L_\mu\not=0$. It then follows from Lemma~\ref{Qk} that~$Q_{\KKK}(\mu)$ satisfies the property of~$P_{\KKK}(\mu)$ in (i)(a). If there are other~$\nu\in\KKK$ such that~$L(\nu)_\mu\not=0$, then {these} are only finitely many. By induction we can assume that we already constructed $P_{\KKK}(\nu)$ for all {such $\nu$}. It follows that all these are {isomorphic to} direct summands ({with respective multiplicities} $\dim L(\nu)_\mu$) of~$Q_{\KKK}(\mu)$. 
The remaining direct summand of~$Q_{\KKK}(\mu)$ satisfies the properties of~$P_{\KKK}(\mu)$ in (i)(a).

By Corollary~\ref{CorAdd}, the module ${P}_{\KKK}(\mu)$ is {an object of}~$\cF^\Delta$.  
Lemma~\ref{LemDN} implies that for any $\nu\in\fh^\ast$
$$(P_{\KKK}(\mu):\Delta(\nu))\;=\;\dim\Hom_{\fg}(P_{\KKK}(\mu),\nabla(\nu)).$$
If $\nu\in\KKK$, then part (i)(a) {yields}
$$\dim\Hom_{\fg}(P_{\KKK}(\mu),\nabla(\nu))\;=\;[\nabla(\nu):L(\mu)]\;=\; [\Delta(\nu):L(\mu)],$$
where the latter equality follows from the duality $\vee$. This concludes the proof of part (i)(b).

Part (i)(c) follows from part (i)(b) and Corollary~\ref{CorBGGThm}(i).

We consider an arbitrary module $M\in {}^{\KKK}\bO$. It has a set of generating elements~$\{v_\alpha\}\subset M$, which we can choose to be weight vectors. Since $M\in {}^{\KKK}\bO$, it follows {that} $U(\fg)v_\alpha$ is  {isomorphic to} a quotient of~$Q_{\KKK}(\mu)$ {where} $\mu\in\fh^\ast$ {is} the weight of~$v_\alpha$. Hence, {there is} an epimorphism $\bigoplus_\mu Q_{\KKK}(\mu)\tto M$. From the universality property of coproducts, or alternatively from Lemma~\ref{LemCopr}, it follows that~$\bigoplus_\mu Q_{\KKK}(\mu)$ is projective. This proves part (ii).
\end{proof}

\subsection{Category $\bO$ as a Grothendieck category}\label{GroCat}

An object $G$ in an abelian category $\cC$ is a generator if the functor $\Hom_{\cC}(G,-):\cC\to\Ab$ is faithful. Following \cite{KS}, a Grothendieck category is an abelian category which admits set valued direct sums and a generator{,} and in which direct limits of short exact sequences are exact. By \cite[Theorem~9.6.2]{KS}, Grothendieck categories have enough injective objects.
 
\begin{prop}\label{PropGroth}
The category $\bO$ is a Grothendieck category. In particular, $\bO$ has enough injective objects.
\end{prop}
First we prove the following lemma, which will be useful in the following sections as well.

\begin{lemma}\label{LemQO}
For each $M\in\bO$, $\mu\in\fh^\ast$ and $v\in M_\mu$, there exists $\KKK\in\cK$ such that $v$ {lies} the image of a morphism $Q_{\KKK}(\mu)\to M$.
\end{lemma}
\begin{proof}
By definition, the space $U(\fb)v$ is finite dimensional and hence, we can construct $\KKK\in\cK$ with $\supp U(\fb)v\subset \KKK$. The submodule of $M$ generated by $v$ is thus {an object of} ${}^{\KKK}\bO${,} and the result follows from Lemma~\ref{Qk}.
\end{proof}

\begin{proof}[Proof of Proposition~\ref{PropGroth}]
It follows from the definition that $\bO$ admits arbitrary direct sums. Direct limits of short exact sequences are exact as this is true for vector spaces. Finally, by Lemma~\ref{LemQO}, the $\fg$-module
$$G:=\bigoplus_{\mu\in\fh^\ast,\KKK\in\cK}Q_{\KKK}(\mu)$$
 is a generator in $\bO$.
\end{proof}

\begin{cor}\label{CorDL}
For $M\in\bO$ and $\KKK\in\cK$, denote by ${}^{\KKK}M$ the maximal submodule of $M$ in ${}^{\KKK}\bO$.
Then we have $M\cong \varinjlim ({}^{\KKK}M)$.
\end{cor} 
\begin{proof}
We only need to prove that the canonical inclusion $\varinjlim ({}^{\KKK}M)\subset M$ is an equality. The latter follows from the fact that any weight vector in $M$ is included in a submodule of $M$ in ${}^{\KKK}\bO$, indeed we can take that submodule to be the image of $Q_{\KKK}(\mu)$ under the morphism in Lemma~\ref{LemQO}.
\end{proof}

\subsubsection{} The analogue of Theorem~\ref{ThmProj} for injective objects holds. Concretely, we can define the injective cover $I_{\KKK}(\mu)=P_{\KKK}(\mu)^\vee$ of $L(\mu)$ in ${}^{\KKK}\bO$. That $I_{\KKK}(\mu)$ is injective follows for instance from first defining $J_{\KKK}(\mu)=Q_{\KKK}(\mu)^\vee$. Indeed, the fact that any $M\in {}^{\KKK}\bO$ is the union of its finitely generated submodules $\{M^\alpha\}$, where the latter belong to $\bar{\cO}$, and {the isomorphisms}
$$\Hom_{\fg}(M,J_{\KKK}(\mu))\;\cong\;\varprojlim \Hom_{\fg}(M^\alpha,J_{\KKK}(\mu))\;\cong\;\varprojlim (M^\alpha_\mu)^\ast\;\cong\; (M_\mu)^\ast$$
show that $J_{\KKK}(\mu)$ is injective. By construction, $I_{\KKK}(\mu)$ is a direct summand of $J_{\KKK}(\mu)$.
\begin{prop}
The injective hull of $L(\mu)$ in $\bO$ is given by $I(\mu):=\varinjlim_{\KKK\in\cK} I_{\KKK}(\mu)$.\end{prop}
\begin{proof}
If $I$ is the injective hull of $L(\mu)$, which exists by Proposition~\ref{PropGroth}, then clearly $^{\KKK}I=I_{\KKK}(\mu)$ for all $\KKK\in\cK$. The proposition thus follows as a special case of Corollary~\ref{CorDL}.
\end{proof}

\subsection{Blocks}
For $\lambda\in\fh^\ast$, let $\bO_{\lbra \lambda\rbra}$ denote the full subcategory of modules $M$ in~$\bO$ such that~$[M:L(\mu)]=0$ whenever $\mu\not\in\lbra\lambda\rbra$. We use similar notation for $\bar{\cO}$ and~$\cO$.

\begin{prop}\label{PropBlocks}
{There is} an equivalence of categories
$$\prod_{\lbra \lambda\rbra}\bO_{\lbra \lambda\rbra}\;\stackrel{\sim}{\to}\; \bO,\qquad (M_{\lbra\lambda\rbra})_{\lbra\lambda\rbra}\mapsto\bigoplus_{\lbra\lambda\rbra}M_{\lbra\lambda\rbra}.$$
\end{prop}
\begin{proof}
By definition,
$$\Hom_{\prod_{\lbra \lambda\rbra}\bO_{\lbra \lambda\rbra}}\left((M_{\lbra\nu\rbra})_{\lbra\nu\rbra},(N_{\lbra\mu\rbra})_{\lbra\mu\rbra}\right)\;=\;\prod_{\lbra\lambda\rbra}\Hom_{\bO_{\lbra\lambda\rbra}}(M_{\lbra\lambda\rbra},N_{\lbra\lambda\rbra}).$$
On the other hand, by \eqref{eqCo},
$$\Hom_{\bO}(\bigoplus_{\lbra\lambda\rbra}M_{\lbra\lambda\rbra},\bigoplus_{\lbra\mu\rbra}N_{\lbra\mu\rbra})\;\cong\; \prod_{\lbra\lambda\rbra}\Hom_{\bO}(M_{\lbra\lambda\rbra},\bigoplus_{\lbra\mu\rbra}N_{\lbra\mu\rbra})\;\cong\; \prod_{\lbra\lambda\rbra}\Hom_{\bO}(M_{\lbra\lambda\rbra},N_{\lbra\lambda\rbra}).$$
Hence, the functor $\prod_{\lbra \lambda\rbra}\bO_{\lbra \lambda\rbra}\to \bO$ is fully faithful. We denote the isomorphism closure of its image by $\bO'$, which is a subcategory closed under taking quotients. The generator $G$ of $\bO$ in the proof of Proposition~\ref{PropGroth} is included in $\bO'$. This shows that $\bO'=\bO$.
\end{proof}

\begin{rem}
For any ideal $\KKK\subset\fh^\ast$ and~$\lambda\in\KKK$, we use the notation ${}^{\KKK}\bO_{\lbra\lambda\rbra}$ for the full subcategory of $\bO$ consisting of modules which are both in~${}^{\KKK}\bO$ and~$\bO_{\lbra\lambda\rbra}$. It is clear, by equation~\eqref{eqDtrunc}, that ${}^{\KKK'}\bO_{\lbra\lambda\rbra}={}^{\KKK}\bO_{\lbra\lambda\rbra}$ for two ideals $\KKK$ and~$\KKK'$ such that~$\KKK\cap\lbra\lambda\rbra=\KKK'\cap\lbra\lambda\rbra$.
\end{rem}

\begin{rem}
Proposition~\ref{PropBlocks} implies in particular \cite[Theorem~3.4]{Nam}, which was obtained through a different approach.  Note however that we have proper inclusions of categories
$$\bigoplus_{\lbra \lambda\rbra}\bar{\cO}_{\lbra \lambda\rbra}\;\subsetneq\;\bar{\cO}\;\subsetneq\; \prod_{\lbra \lambda\rbra}\bar{\cO}_{\lbra \lambda\rbra}.$$
\end{rem}

\subsection{Describing algebras}
Fix an upper finite ideal $\KKK\subset\fh^\ast$ and~$\lambda\in \KKK$. 

\subsubsection{}\label{DefAlgA}We set ${}^{\KKK}\lbra\lambda\rbra=\KKK\cap \lbra\lambda\rbra$. We define the vector space
$$A^{\KKK}_{\lbra\lambda\rbra}(\fg,\fb)=A^{\KKK}_{\lbra\lambda\rbra}\;:=\; \bigoplus_{\mu,\nu\in {}^{\KKK}\lbra\lambda\rbra}\Hom_{\fg}(P_{\KKK}(\mu), P_{\KKK}(\nu)),$$
which is an algebra with multiplication $fg=g\circ f$. The algebra $A^{\KKK}_{\lbra\lambda\rbra}$ is then locally finite, with idempotents~$e_\nu$ given by the identity of $P_{\KKK}(\nu)$, for all $\nu\in\KKK$.

\begin{thm}\label{ThmAMod}
{There is} an equivalence of categories
$${}^{\KKK}\bO_{\lbra\lambda\rbra}\;\stackrel{\sim}{\to}\; A^{\KKK}_{\lbra\lambda\rbra}\mbox{{\rm -Mod}},\quad M\mapsto \bigoplus_{\mu\in {}^{\KKK}\lbra\lambda\rbra}\Hom_{\fg}(P_{\KKK}(\mu),M).$$
\end{thm}
\begin{proof}
We write $A=A^{\KKK}_{\lbra\lambda\rbra}$ and~$\Ffun=\bigoplus_{\mu}\Hom_{\fg}(P_{\KKK}(\mu),-)$.
We observe that
$$\Ffun(P_{\KKK}(\nu))\;\cong\; Ae_{\nu}\quad\mbox{for all $\nu\in\KKK$.}$$
Furthermore, $\Ffun$ induces an isomorphism
$$\Hom_{\fg}(P_{\KKK}(\nu),P_{\KKK}(\kappa))\;\stackrel{\sim}{\to}\; \Hom_A(Ae_{\nu},Ae_{\kappa})\quad\mbox{for all $\nu,\kappa\in\KKK$.}$$

It is clear that the functor~$\Ffun$ preserves arbitrary coproducts. Hence, $\Ffun$ restricts to an equivalence between the categories of projective objects in~${}^{\KKK}\bO$ and~$A$-Mod. The fact that the functor~$\Ffun$ is exact then implies that $\Ffun$ is an equivalence of categories between  ${}^{\KKK}\bO$ and~$A$-Mod.
 \end{proof}


\begin{rem}\label{RemDomA}
{If} $\lambda\in\fh^\ast$ is dominant we can take $\KKK:=\{\mu\in\fh^\ast\,|\, \mu\le\lambda\}$, in which case we have ${}^{\KKK}\bO_{\lbra\lambda\rbra}=\bO_{\lbra\lambda\rbra}$. It thus follows that~$\bO_{\lbra\lambda\rbra}$ is equivalent to the category of modules over a locally finite associative algebra.
\end{rem}

\begin{cor}\label{CorAmod}
{We have} an equivalence of categories
$${}^{\KKK}\bar{\cO}_{\lbra\lambda\rbra}\;\stackrel{\sim}{\to}\; A^{\KKK}_{\lbra\lambda\rbra}\mbox{{\rm -mod}},\quad M\mapsto \bigoplus_{\mu\in {}^{\KKK}\lbra\lambda\rbra}\Hom_{\fg}(P_{\KKK}(\mu),M).$$
\end{cor}
\begin{proof}
{T}he equivalence of Theorem~\ref{ThmAMod} {induces an equivalence of the respective} Serre subcategories of modules with finite multiplicities, by Lemma~\ref{LemKmult}.
\end{proof}

\begin{lemma}
For $\KKK\subset \KKK'\in\cK$, we have an isomorphism 
$$A^{\KKK'}_{\lbra\lambda\rbra}/I\;\cong\; A^{\KKK}_{\lbra\lambda\rbra}\quad\mbox{with}\quad I:=\sum_{\mu\in {}^{\KKK'}\lbra\lambda\rbra\backslash{}^{\KKK}\lbra\lambda\rbra }A^{\KKK'}_{\lbra\lambda\rbra}e_\mu A^{\KKK'}_{\lbra\lambda\rbra}.$$
\end{lemma}
\begin{proof}
Consider the short exact sequence 
$$0\to N(\nu)\to P_{\KKK'}(\nu)\stackrel{p_\nu}{\to} P_{\KKK}(\nu)\to0,$$ for any $\nu\in\KKK$. For all $\nu,\kappa\in\KKK$, 
we have
$$\Hom_{\bC}(N(\nu),P_{\KKK}(\kappa))=0\quad\mbox{and}\quad \Ext^1_{\bC}(P_{\KKK'}(\nu),N(\kappa))=0.$$
Studying the long exact sequences coming from the bifunctor $\Hom_{\bC}(-,-)${,} acting on short exact sequences as above{,} then yields an epimorphism and isomorphism
$$\xymatrix{
\Hom_{\bC}(P_{\KKK'}(\nu),P_{\KKK'}(\kappa))\ar@{->>}[rr]^{p_{\kappa}\circ-}&&\Hom_{\cC}(P_{\KKK'}(\nu),P_{\KKK}(\kappa))&&\ar[ll]^{-\circ p_\nu}_{\sim}\Hom_{\cC}(P_{\KKK}(\nu),P_{\KKK}(\kappa)).
\\
}$$
Composing the epimorphism and the inverse of the isomorphism {gives} the epimorphism
$$\Hom_{\bC}(P_{\KKK'}(\nu),P_{\KKK'}(\kappa))\tto \Hom_{\cC}(P_{\KKK}(\nu),P_{\KKK}(\kappa)),\quad \alpha\mapsto \phi,\quad\mbox{when $p_\kappa\circ\alpha=\phi\circ p_\nu$}.$$
This clearly yields an algebra {epi}morphism ${\varepsilon\colon}A^{\KKK'}_{\lbra\lambda\rbra}\tto A^{\KKK}_{\lbra\lambda\rbra}$.

{A} projective cover of $N(\kappa)$ is {isomorphic to} a direct sum of modules $P_{\KKK'}(\mu)$ with $\mu\in {}^{\KKK'}\lbra\lambda\rbra\backslash{}^{\KKK}\lbra\lambda\rbra$. Any $\alpha: P_{\KKK'}(\nu)\to P_{\KKK'}(\kappa)$ with $p_\kappa\circ\alpha=0$ thus factors through such a projective module. This shows that the ideal $I$ is the kernel of the epimorphism {$\varepsilon$}.
\end{proof}

\begin{prop}\label{FDprop}
Assume that~$\fg$ is finite dimensional, and hence a reductive Lie algebra. 
\begin{enumerate}[(i)]
\item We have equivalences of categories
$$\bO\;\cong\; \prod_{\lbra\lambda\rbra} \bO_{\lbra\lambda\rbra}\quad\mbox{and}\quad \cO\;\cong\;\bigoplus_{\lbra\lambda\rbra} \cO_{\lbra\lambda\rbra}. $$
\item For each $\lambda\in\fh^\ast$, there exist finite dimensional algebra $A$ such that
$$\cO_{\lbra\lambda\rbra}\;\cong\; A\mbox{{\rm -mod}}\quad\mbox{and}\quad \bO_{\lbra\lambda\rbra}\;\cong\; A\mbox{{\rm -Mod}}.$$
\item The subcategory $\cO$ is extension full in~$\bO$.
\end{enumerate}
\end{prop}
\begin{proof}
Part (i) is a special case of Proposition~\ref{PropBlocks}, see also \cite{BGG}. Part (ii) follows from Theorem~\ref{ThmAMod} and Corollary~\ref{CorAmod}, by observing that~$\bar{\cO}_{\lbra\lambda\rbra}=\cO_{\lbra\lambda\rbra}$, see e.g.~\cite[Section~5.1]{NamT}.
Part (iii) follows immediately from observing that a minimal projective resolution in~$\bO$ of $M\in\cO$, is actually in~$\cO$.
\end{proof}

\subsection{Extension fullness}

\begin{thm}\label{ThmExtFull}
Let $\KKK\subset\fh^\ast$ be an upper finite ideal.
\begin{enumerate}[(i)]
\item The category ${}^{\KKK}\bO$ is extension full in~$\bO$.
\item The category ${}^{\KKK}\bO$ is extension full in~$\bC$.
\item For the inclusion $\iota:{}^{\KKK}\bar{\cO}\hookrightarrow {}^{\KKK}\bO$, the morphism $\iota^i_{MN}$ in \eqref{XYi} is an isomorphism for all $i\in\mN$, if $M\in\cF^\Delta$ or $N\in \cF^{\nabla}$.
\end{enumerate}
\end{thm}
Before proving the theorem, we note the following special case.
\begin{cor}
For all~$\mu,\nu\in \fh^\ast$, we have
$$\Ext^i_{\bO}(\Delta(\mu),\nabla(\nu))=0\quad\mbox{for all $i>0$.}$$
\end{cor}
\begin{proof}
Assume first that~$\nu\not>\mu$ and let $\KKK$ be the ideal generated by~$\mu$ and~$\nu$. Then $\Delta(\mu)=P_{\KKK}(\mu)${.} {Hence} $\Ext^i_{{}^{\KKK}\bO}(\Delta(\mu),\nabla(\nu))=0${,} and the conclusion follows from Theorem~\ref{ThmExtFull}(i). If $\nu>\mu$, we can use $\nabla(\nu)=I_{\KKK}(\nu)$, or alternatively Theorem~\ref{ThmExtFull}(iii) and the duality $\vee$.
\end{proof}

We start the proof of the theorem with the following proposition.

\begin{prop}\label{PropSumD}
For $\KKK\in\cK$ and a family $\{M_\alpha\}$ of objects in $\cF^\Delta[\fh^\ast\backslash \KKK]$, set $M=\bigoplus_\alpha M_\alpha$. {Then}
$$\Ext^k_{\bO}(M,N)=0\quad\mbox{for all $N\in{}^{\KKK}\bO$ and~$k\in\mN$.}$$
\end{prop}
\begin{proof}
The case $k=0$ is obvious {from} equation~\eqref{eqCo}. The case $k=1$ follows from Lemma~\ref{LemCopr} and Corollary~\ref{CorExt1}. {Next,} we take $k>1$ and assume the proposition is proved for $k-1$. An element {of} $\Ext^k_{\bO}(M,N)$ can be represented by the upper exact sequence in the diagram
$$\xymatrix{
0\ar[r]& N\ar@{=}[d]\ar[r]& E_k\ar[r]& E_{k-1}\ar[r]&\cdots\ar[r]& E_2\ar[r]& E_1\ar[r]& M\ar[r]\ar@{=}[d]& 0\\
0\ar[r]& N\ar[r]& E_k'\ar[u]\ar[r]& E_{k-1}'\ar[u]\ar[r]&\cdots\ar[r]& E_2'\ar[u]\ar[r]& \bigoplus_\alpha P^\alpha\ar[r]\ar[u]& M\ar[r]& 0.
}$$
Now $M_\alpha\in \cF^\Delta$ is generated by finitely many weight vectors, say of weights $\{\mu^\alpha_j\}_j$ in $\fh^\ast\backslash\KKK$, and we take a weight vector in~$E_1$ in the preimage of each such generating weight vector. By Lemma~\ref{LemQO} the submodule of $E_1$ generated by those weight vectors is {isomorphic to} a quotient of a finite direct sum $P^\alpha:=\bigoplus_j Q_{\KKK^\alpha_j}(\mu^\alpha_j)$ for upper finite ideals $\KKK^\alpha_j\ni \mu^\alpha_j$. Using pull-backs we thus arrive at the above commutative diagram with exact rows. By Corollary~\ref{resolv}, the kernel $K^\alpha$ of $P^\alpha\tto M_\alpha$ is {an object of}~$\cF^\Delta$. By construction, $K^\alpha$ is {actually an object of}~$\cF^\Delta[\fh^\ast\backslash \KKK]$. By induction, the 
 {element in $\operatorname{Ext}^{k-1}_{{\bf O}}(\bigoplus_\alpha K^\alpha, N)$ represented by the sequence}
$$\xymatrix{
0\ar[r]& N\ar[r]& E_k'\ar[r]& E_{k-1}'\ar[r]&\cdots\ar[r]& E_2'\ar[r]& \bigoplus_\alpha K^\alpha \ar[r]& 0
}$$
{equals} zero. It follows that the {element in $\operatorname{Ext}^k_{{\bf O}}(M,N)$ represented by the above diagram also equals zero}.
\end{proof}


\begin{lemma}\label{CorNPM} Let $\KKK\subset\fh^\ast$ be an upper finite ideal and~$\SSS$ a coideal in~$\KKK$ (for instance $\SSS=\KKK$).
For any $M\in \cF^{\Delta}[\SSS]$, we have a short exact sequence in  $\cF^{\Delta}[\SSS]$
$$0\to X\to \bigoplus_{\mu}P_{\KKK}(\mu)\to M\to 0.$$
\end{lemma}
\begin{proof}
This is an immediate consequence of the structure of projective objects in ${}^{\KKK}\bO$ and {of} Corollary~\ref{resolv}.
\end{proof}

\begin{proof}[Proof of Theorem~\ref{ThmExtFull}]
For part (i), { [CM1, Corollary 5] shows that} it suffices to prove that{,} {if $i>0$ then}~$\Ext^i_{\bO}(P,N)=0$ {for any} $N\in {}^{\KKK}\bO$ and~ {any projective object} $P\in{}^{\KKK}\bO$.
 Using the main method in the proof of Proposition~\ref{PropSumD}, one shows that~$\Ext^i_{\bO}(P,N)\not=0$ implies~$\Ext^{i-1}_{\bO}(M,N)\not=0$ for some direct sum {$M$}  of modules in~$\cF^{\Delta}[\fh^\ast\backslash\KKK]$. This contradicts Proposition~\ref{PropSumD}.

Part~(ii) can also be proved by an application of \cite[Corollary~5]{CM}. It follows from the construction in Section~\ref{SecConstProj} that each projective object $P$ in ${}^{\KKK}\bO$ has a resolution by direct sums of modules $U(\fg)\otimes_{U(\fh)}\mk_\kappa$, with $\kappa\not\in\KKK$ outside of position $0$ of the resolution. This implies that $\Ext^i_{\bC}(P,N)=0$ for all $i>0$ and~$N\in{}^{\KKK}\bO$.

By Lemma~\ref{CorNPM}, any module in~$\cF^{\Delta}[\KKK]$ has a projective resolution in~${}^{\KKK}\bO$ which is actually contained in~${}^{\KKK}\bar{\cO}$.
 Part (iii) follows from this observation by applying the duality $\vee$ on $\cC$.
\end{proof}

\begin{cor}\label{CornCoh}
For arbitrary $i\in\mN$, $\mu\in\fh^\ast$ and $M\in\bO$ with upper finite $\supp M $, we have
$$\Ext^i_{\bO}(\Delta(\mu),M)\;\cong\;\Hom_{\fh}(\mk_\mu,{\mathrm H}^i(\fn^+,M)).$$
\end{cor}
\begin{proof}
Let $\KKK$ be the ideal generated by $\supp M $ and $\mu$. Theorem~\ref{ThmExtFull}(i) and (ii) imply
$$\Ext^i_{\bO}(\Delta(\mu),M)\;\cong\;\Ext^i_{{}^{\KKK}\bO}(\Delta(\mu),M)\;\cong\;\Ext^i_{\bC}(\Delta(\mu),M).$$
The reformulation in terms of $\fn^+$-Lie algebra cohomology then follows as in the finite dimensional case, see e.g. \cite[Theorem~6.15(b)]{Humphreys} or \cite[Corollary~14]{CM}.
\end{proof}

\begin{que}\label{QueExtFull}
\begin{enumerate}[(i)]
\item Is $\bO$ extension full in~$\bC$?
\item Is $\bar{\cO}$ extension full in~$\bO$?
\end{enumerate}
\end{que}
Note that Question~\ref{QueExtFull}(i) has a positive answer when restricting to the block $\bO_{\lbra\lambda\rbra}$, for a dominant $\lambda$, by Remark~\ref{RemDomA} and Theorem~\ref{ThmExtFull}(ii).

\subsection{Extensions in Serre quotients}For any ideal $\LLL$, we denote {by~${}_{\LLL}\bO$} the Serre quotient $\bO/{}^{\LLL}\bO$, see Appendix~\ref{AppSerre}. 
\begin{prop}\label{PropExtpi}
Let $\LLL\subset\KKK\subset\fh^\ast$ be ideals and let $\KKK$ be upper finite. For $M\in\cF^{\Delta}[\KKK\backslash \LLL]$ and~$N\in{}^{\KKK}\bO$, we have isomorphisms
$$\pi:\Ext^i_{{}^{\KKK}\bO}(M,N)\;\;\tilde\to\;\; \Ext^i_{{}_{\LLL}^{\KKK}\bO}(M,N)\quad\mbox{for all $i\in\mN$.}$$
\end{prop}
\begin{proof}
First we consider the case $i=0$. Clearly{,}~$M\in\cF^{\Delta}[\KKK\backslash \LLL]$ has no proper submodule~$M'$ such that~$M/M'$ is in~${}^{\LLL}\bO$, hence
$$\Hom_{{}_{\LLL}\bO}( M,N)\;=\;\varinjlim \Hom_{\bO}(M,N/N'),$$
where~$N'$ runs over all submodules of~$N$ which are in~${}^{\LLL}\bO$.
For such $N'$, the exact sequence
$$\Hom_{\bO}(M,N')\to \Hom_{\bO}(M,N)\to\Hom_{\bO}(M,N/N')\to \Ext^1_{\bO}(M,N')$$
has first and last term equal to zero, see Lemma~\ref{CorExt1}. Consequently
all the maps in the direct limit are isomorphisms. Hence, we find an isomorphism
$$\pi:\Hom_{\bO}(M,N)\;\;\tilde\to\;\; \Hom_{{}_{\LLL}\bO}(M,N).$$

Lemma~\ref{CorNPM} implies that, inside ${}^{\KKK}\bO$ the module $M$ has a projective resolution $P_\bullet$ with $P_i\in\cF^{\Delta}(\KKK\backslash\LLL)$ for all $i\in\mN$. By Lemma~\ref{ProjBC}, $\pi(P_\bullet)$ is a projective resolution of $M$ in ${}^{\KKK}_{\LLL}\bO$. Since the extension groups are then calculated as ${\mathrm H}^i(\Hom(P_\bullet,N))$ in the respective categories, the conclusion follows from the above paragraph.
\end{proof}
 

\begin{rem}
For an upper finite ideal $\KKK\subset\fh^\ast$ and an arbitrary ideal $\LLL\subset \KKK$, the category ${}^{\KKK}_{\LLL}\bO$ has enough projective objects by Theorem~\ref{ThmProj} and Lemma~\ref{ProjBC}.

This observation extends to the case of ideals $\LLL\subset \KKK\subset \fh^\ast$ such that~$\KKK\backslash\LLL$ is upper finite in~$\fh^\ast\backslash\LLL$, using Remark~\ref{RemBruhat}(ii).
\end{rem}

\subsection{Example: $\mathfrak{gl}_{\infty}$}
For $\mathfrak{gl}_{\infty}$ we have precisely two Dynkin Borel subalgebras~$\fb$ and~$\fb'$, up to conjugacy. We consider a Cartan subalgebra $\fh$ contained in both {$\mathfrak{b}$ and $\mathfrak{b}'$}.
\subsubsection{}
For $\fg\supset\fb\supset\fh$, we choose a realisation where
$$\rR=\{\epsilon_i-\epsilon_j\,|\, i\not=j\in\mN\} \quad\mbox{and }\quad \rR^+=\{\epsilon_i-\epsilon_j\,|\, i<j\}.$$
For $\fg\supset\fb'\supset\fh$, we choose a realisation where
$$\rR=\{\epsilon'_i-\epsilon'_j\,|\, i\not=j\in\mZ\} \quad\mbox{and }\quad \rR^+=\{\epsilon'_i-\epsilon'_j\,|\, i<j\}.$$
We show that these lead to different theories in the following sense.
\begin{lemma}
There exists no equivalence of categories $\bO_{\lbra 0\rbra }(\fg,\fb)\to \bO_{\lbra 0\rbra }(\fg,\fb')$, or $\bar\cO_{\lbra 0\rbra }(\fg,\fb)\to \bar\cO_{\lbra 0\rbra }(\fg,\fb')$ which exchanges Verma modules. 
\end{lemma}
\begin{proof} We set $W=W(\fg:\fh)$.  {D}enote the simple reflections with respect to $\fb$ by $s_i=r_{\epsilon_i-\epsilon_{i+1}}\in W$ for $i\in\mN${,} and with respect to $\fb'$ by~$s_i'=r_{\epsilon'_i-\epsilon'_{i+1}}\in W$ for $i\in\mZ$.
 {Let $\uparrow$ denote} the Bruhat order corresponding to $\fb$, {and let $\uparrow'$ be the corresponding order for $\fb'$}.
By Corollary~\ref{CorBGGThm}(ii), it suffices to prove that the partially ordered sets $(W\cdot 0,\uparrow )$ and~$(W\cdot 0,\uparrow' )$ are not isomorphic. 

To look for a contradiction, we assume that we have an isomorphism of posets $\phi :(W\cdot 0,\uparrow )\to (W\cdot 0,\uparrow' )$.
Both posets have a unique maximal element, $0$, which must be exchanged by $\phi$. The sets of elements~$\mu$ which are {immediate predecessors of} $0$ {in terms of the  orders $\uparrow$ and $\uparrow'$} must also be exchanged by $\phi$. We must thus have a bijection
$$\phi: \{s_i\cdot 0,i\in\mN \}\;\to\;  \{s_i'\cdot 0,i\in\mZ \}.$$

Now we define $c(\mu)\in\mN$ for any of the weights $\mu$ in~$\{s_i\cdot 0,i\in\mN \}$ as the number of other elements~$\nu$ in that set such that there exist (at least) $2$ elements in~$W\cdot0$ which are covered by both $\mu$ and~$\nu$. We find that~$c(s_0\cdot 0)=1$ because {$\nu:=s_1\cdot 0$ is the only weight such that $\mu=s_0\cdot 0$  and $\nu$} cover two weights, namely $s_0s_1\cdot 0$ and~$s_1s_0\cdot 0$. We have $c(s_i\cdot 0)=2$ for $i>0$, coming from~$s_{i-1}\cdot0$ and~$s_{i+1}\cdot0$. Similarly, the same definition { for the set $\{s'_i\cdot 0, i\in\mathbb{Z}\}$ yields} $c(s'_i\cdot 0)=2$ for all $i\in\mZ$. This contradicts the existence of $\phi$.
\end{proof}

We conclude this section with an example of an infinite dimensional extension space in~$\bO$ for simple modules.
\begin{lemma}
For $\fg=\mathfrak{gl}_{\infty}$ and both of the Dynkin Borel subalgebras {$\fb$ and $\fb'$}, we have
$$\dim_{\mk}\Ext^2_{\bO}(\mk,\mk)\,=\,\infty.$$
\end{lemma}
\begin{proof}
{By} Theorem~\ref{ThmExtFull}(i){,} we can equivalently calculate  {$\operatorname{Ext}^2_{{}^{\KKK}\bO}(\mk,\mk)$} for some $0\in\KKK\in\cK$. Theorem~\ref{ThmExtFull}(ii) then shows we can {instead} calculate  {$\operatorname{Ext}^2_{\bC}(\mk,\mk)$}{, and this {is} what we will do.}

 Taking the standard projective resolution of $\mk$ in~$\bC$ shows that~$\Ext^\bullet_{{\bC}}(\mk,\mk)$ is the cohomology of the complex
$$0\to\mk \to \Hom_{\fh}(\fg/\fh,\mk)\to \Hom_{\fh}(\Lambda^2{(}\fg/\fh{)},\mk)\to  \Hom_{\fh}(\Lambda^3{(}\fg/\fh{)},\mk)\to\cdots.$$ 
Since $\Hom_{\fh}(\fg/\fh,\mk)=0$, we find the well-known properties $\Hom_{\bC}(\mk,\mk)=\mk$, $\Ext^1_{{\bC}}(\mk,\mk)=0$ and also that~$\Ext^2_{{\bC}}(\mk,\mk)$ is the kernel of $ \Hom_{\fh}(\Lambda^2{(}\fg/\fh{)},\mk)\to  \Hom_{\fh}(\Lambda^3{(}\fg/\fh{)},\mk)$. That kernel is easily seen to be infinite dimensional.
\end{proof}


\section{Link with finite case}

\subsection{Induction and restriction functors}

\subsubsection{} For each $n\in\mN$ and~$\lambda\in\fh^\ast$, we have the complete {sub}set $\Lambda_n:=[\lambda]_n=\lambda+\mZ\rR_n$ {of}~$\fh^\ast${, and} the corresponding ideals $\mathring{\Lambda}_n$ and~$\overline{\Lambda}_n$ as in \ref{Complete2}.
The exact functors of Section~\ref{IndRes} restrict to exact functors
$$\ind^{n}_+=\ind^{\fg}_{\fg_n,+}:{}^{\Lambda_n}\bO(\fg_n,\fb_n)\to{}^{\overline{\Lambda}_n}\bO,\;\quad \res^n_\lambda=\res^{\fg}_{\fg_n,\lambda}: {}^{\overline{\Lambda}_n}\bO\to {}^{\Lambda_n}\bO(\fg_n,\fb_n).$$

The following is an infinite rank version of \cite[Theorem~32]{CMZ}. We provide an alternative proof. 
\begin{thm}\label{ThmEquiv}
For $n\in\mN$ and~$\lambda\in\fh^\ast$, {there are} mutually inverse equivalences of abelian categories~$\Psi$ and~$\Phi$, admitting a {commutative} diagram
$$
\xymatrix{
{}^{\Lambda_n}\bO({\fg_n},\fb_n)\ar[rr]^{\ind^n_+}\ar[drr]^{\Psi}&&{}^{\overline{\Lambda}_n}\bO\ar[rr]^{\res_\lambda^n}\ar[d]^{\pi}&&{}^{\Lambda_n}\bO({\fg_n},\fb_n)\\
&&{}^{\overline{\Lambda}_n}_{\mathring{\Lambda}_n}\bO\ar[rru]^{\Phi}
}{.}
$$
Moreover, for any $\mu\in {\Lambda_n}$, {there is an isomorphism} $\Psi(\Delta_n(\mu))\cong\Delta(\mu)$ in~${}^{\overline{\Lambda}_n}_{\mathring{\Lambda}_n}\bO$.
\end{thm}
\begin{proof}
 We define $\Psi:=\pi\circ\ind^n_+$. The existence of a functor~$\Phi$ which completes the {commutative} diagram follows from Lemma~\ref{factorspi} in Appendix~\ref{AppSerre}. By construction, {there is an isomorphism of functors} $\res_\lambda^n\circ\ind^n_+\cong\id_{{}^{\Lambda_n}\bO({\fg_n},\fb_n)}$. By commutativity of the diagram, we then have {an isomorphism of functors} $\Phi\circ\Psi\cong\id_{\bO_{{\Lambda_n}}({\fg_n},\fb_n)}$. To conclude the proof it suffices to show that~$\Phi$ is faithful. 

We will write $\cA:={}^{\overline{\Lambda}_n}_{\mathring{\Lambda}_n}\bO$.
Consider $M,N\in {}^{\overline{\Lambda}_n}\bO$ and~$f\in \Hom_{\cA}(M,N)$. 
By Lemma~\ref{factorspi}, we have
$\Phi(f)=\res^{\Lambda_n}_n(g)$, for any representative $g\in \Hom_{\fg}(M',N/N')$ of~$f$, where~$M'\subset M$ { and $N'\subset N$ satisfy $\supp (M/M')\subset{\mathring{\Lambda}_n}$, $\supp N'\subset {\mathring{\Lambda}_n}$}. 

Now assume $\Phi(f)=0$, which thus
 implies that~$g$ restricted to the weight spaces of~$M'$ for weights in~${\Lambda_n}$ {equals} zero. Since $g$ is a morphism of~$\fh$-modules{,} this means that the image of~$g$ is of the form $N''/N'$ for some $N''\supset N'$ with $\supp N''\subset \mathring{\Lambda}_n$. Thus {the element corresponding to $g$ }
 in the direct limit defining $\Hom_{\cA}(M,N)$ { equals zero}. {Consequently,}~$f=0$. We hence find that~$\Phi$ is indeed faithful. 
 \end{proof}

\begin{cor}\label{CorEqBlocks} We use the notation of Theorem~\ref{ThmEquiv}.
\begin{enumerate}[(i)]
\item For dominant $\lambda\in\fh^\ast$, the functor $\Psi$ restricts to an equivalence between $\bO_{\lbra\lambda\rbra_n}(\fg_n,\fb_n)$ and the Serre quotient of $\bO_{\lbra\lambda\rbra}$ with respect to the subcategory with only non-zero multiplicities for simple modules $L(w\cdot\lambda)$ with $w\not\in W_n$.
\item For antidominant $\lambda\in\fh^\ast$, the functor $\Psi$ restricts to an equivalence between $\bO_{\lbra\lambda\rbra_n}(\fg_n,\fb_n)$ and the Serre subcategory of $\bO_{\lbra\lambda\rbra}$ with only non-zero multiplicities for simple modules $L(w\cdot\lambda)$ with $w\in W_n$.
\end{enumerate}
\end{cor}

\subsubsection{}\label{ForEquivK} Consider an arbitrary set $\{\lambda_1,\lambda_2,\cdots,\lambda_k\}\subset\fh^\ast$ and~$n\in\mN$ large enough such that~$\lambda_i-\lambda_j\in\mZ\rR_n$ for all $1\le i,j\le k$. Denote by~$\KKK$, resp. $\KKK_n$, the ideal in~$(\fh^\ast,\le)$, resp. $(\fh^\ast,\le_n)$, generated by~$\{\lambda_1,\lambda_2,\cdots,\lambda_k\}$. The set $\KKK_n$ is complete in~$(\fh^\ast,\le)$, so $\mathring{\KKK}_n:=\KKK\backslash \KKK_n$ is also an ideal in~$(\fh^\ast,\le)$. 

By restricting the equivalence in Theorem~\ref{ThmEquiv}, we obtain the following corollary.

\begin{cor}\label{CorEquiv}
With notation and assumptions as in \ref{ForEquivK}, {there are} mutually inverse equivalences of abelian categories~$\Psi$ and~$\Phi$, admitting a commutative diagram
$$
\xymatrix{
{}^{\KKK_n}\bO({\fg_n},\fb_n)\ar[rr]^{\ind^n_+}\ar[drr]^{\Psi}&&{}^{\KKK}\bO\ar[rr]^{\res_\lambda^n}\ar[d]^{\pi}&&{}^{\KKK_n}\bO({\fg_n},\fb_n)\\
&&{}^{\KKK}_{\mathring{\KKK}_n}\bO\ar[rru]^{\Phi}
}
$$
Moreover, {there is an isomorphism} $\Psi(\Delta_n(\mu))\cong\Delta(\mu)$ in~${}^{\KKK}_{\mathring{\KKK}_n}\bO$ for {any} $\mu\in {\KKK_n}$.
\end{cor}

\begin{thm}\label{Thmidemptr}
With notation and assumptions as in \ref{ForEquivK} and~$\lambda\in\KKK_n$, consider the algebras~$A:=A^{\KKK}_{\lbra\lambda\rbra}$ and~$A_n:=A^{\KKK_n}_{\lbra\lambda\rbra_n}(\fg_n,\fb_n)$ as in \ref{DefAlgA}. For the idempotent $\varepsilon_n=\sum_{\mu}e_\mu\in A$, with $\mu$ ranging over $\KKK_n\cap\lbra\lambda\rbra_n$, we have an algebra isomorphism $\varepsilon_nA\varepsilon_n\cong A_n$.
\end{thm}
\begin{proof}
By Theorem~\ref{ThmAMod}, {there is} an equivalence
$${}^{\KKK_n}\bO_{\lbra\lambda\rbra_n}(\fg_n,\fb_n)\;\cong\;A_n\mbox{-Mod}.$$
By Theorem~\ref{ThmAMod} and Lemma~\ref{LemQuoAlg}, we have equivalences
$${}^{\KKK_n}\bO_{\lbra\lambda\rbra_n}(\fg_n,\fb_n)\;\cong\;{}^{\KKK}_{\mathring{\KKK}_n}\bO\;\cong\;\varepsilon_n A\varepsilon_n\mbox{-Mod}.$$
By construction, both $A_n$ and $\varepsilon_nA\varepsilon_n$ are the endomorphism algebra{s} of {respective maximal} direct sum{s} of {mutually non-isomorphic} indecomposable projective objects in ${}^{\KKK_n}\bO_{\lbra\lambda\rbra_n}(\fg_n,\fb_n)$, implying that {the algebras $A_n$  and $\varepsilon_n A \varepsilon_n$} are isomorphic.
\end{proof}

\begin{cor}
For two integral dominant regular weights $\lambda,\lambda'$, we have an equivalence of categories
$$\bO_{\lbra\lambda\rbra}\;\stackrel{\sim}{\to}\;\bO_{\lbra\lambda'\rbra}\quad\mbox{with}\quad L(w\cdot\lambda)\mapsto L(w\cdot\lambda'),\;\mbox{ for all $w\in W$.}$$
\end{cor}
\begin{proof}
We denote by $\KKK$, resp. $\KKK'$, the ideal in $(\fh^\ast,\le)$ generated by $\lambda$, resp. $\lambda'$. Set $A:=A^{\KKK}_{\lbra\lambda\rbra}$ and $B:=A^{\KKK'}_{\lbra\lambda'\rbra}$. It follows from Theorem~\ref{Thmidemptr} and \cite[Proposition~7.8]{Humphreys} that, for all $n$, {there is} a commuting square of algebra morphisms
$$\xymatrix{
\varepsilon_n A\varepsilon_n\ar@{^{(}->}[rr]\ar[d]^{\sim}&& \varepsilon_{n+1} A\varepsilon_{n+1}\ar[d]^{\sim}\\
\varepsilon_n B\varepsilon_n\ar@{^{(}->}[rr]&& \varepsilon_{n+1}B\varepsilon_{n+1}.
}$$
We thus have $A\cong\varinjlim \varepsilon_n A\varepsilon_n\cong B${,} and the {claimed} equivalence follows.
\end{proof}




\subsection{Extensions of Verma modules}
\begin{thm}\label{ExtDL}
Consider arbitrary $\lambda,\mu\in\fh^\ast$ and~$i\in\mN$. For any $n\in\mN$ such that~$\lambda-\mu \in\mZ\rR_n$, {there is an isomorphism}
$$\Ext^i_{{\bO}}(\Delta(\mu),L(\lambda))\;\cong\; \Ext^i_{\cO(\fg_n,\fb_n)}(\Delta_n(\mu),L_n(\lambda)).$$
\end{thm}
\begin{proof}
Let $\KKK$ be the ideal in~$(\fh^\ast,\le)$ generated by~$\mu$ and~$\lambda$, and~$\KKK_n$ be the ideal in~$(\fh^\ast,\le_n)$ generated by~$\mu$ and~$\lambda$. By Theorem~\ref{ThmExtFull}, it suffices to prove 
$$\Ext^i_{{}^{\KKK}\bO}(\Delta(\mu),L(\lambda))\;\cong\; \Ext^i_{{}^{\KKK_n}\bO(\fg_n,\fb_n)}(\Delta_n(\mu),L_n(\lambda)).$$
By Proposition~\ref{PropExtpi}, the left-hand side is isomorphic to $\Ext^i_{{}^{\KKK}_{\mathring{\KKK}_n}\bO}(\Delta(\mu),L(\lambda))$. The theorem then follows from Corollary~\ref{CorEquiv}.
\end{proof}

In {the} BGG category $\cO(\fg_n,\fb_n)$, the dimensions of the extension spaces $\Ext^i(\Delta_n(\mu),L_n(\lambda))$ are determined by  KLV polynomials. Theorem~\ref{ExtDL} thus shows that the same is true in $\bO$. For instance, let $\mu\in\fh^\ast$ be integral, regular and anti-dominant. With any unexplained notation taken from \cite[Section~8]{Humphreys}, the combination of Theorem~\ref{ExtDL} and \cite[Theorem~8.11(b)]{Humphreys} yields
$$P_{x,w}(q)\;=\;\sum_{i\ge 0}q^i\dim\Ext^{\ell(x,w)-2i}_{\bO}(\Delta(x\cdot\mu),L(w\cdot\mu))\quad\mbox{for all $x,w\in W$},$$ 
{where} $P_{x,w}$ {is} the KLV polynomial corresponding to the Weyl group $W_n$, {and} $n$ {is} big enough so that $x,w\in W_n$. In \cite[Conjecture~8.17]{NamT}, this formula was conjectured for extensions in $\overline{\cO}$. 

\begin{cor}
Conjecture~8.17 in \cite{NamT} is true for ${\bO}$.
\end{cor}
The original question in \cite{NamT} therefore becomes a special case of Question~\ref{QueExtFull}(ii).

\begin{lemma}
Consider arbitrary $\lambda,\mu\in\fh^\ast$ and~$i\in\mN$. For any $n\in\mN$ such that~$\lambda-\mu \in\mZ\rR_n$, we have {an isomorphism}
$$\Ext^i_{{\bO}}(\Delta(\mu),\Delta(\lambda))\;\cong\; \Ext^i_{\cO(\fg_n,\fb_n)}(\Delta_n(\mu),\Delta_n(\lambda)).$$
\end{lemma}
\begin{proof}
Mutatis mutandis Theorem~\ref{ExtDL}.
\end{proof}

\subsection{Standard Koszulity}
{There exists a} notion of graded cover of an abelian category as in Definition~\ref{DefCover}{. In addition, we} refer to Appendix~\ref{AppKos} for  the use of the term ``standard Koszulity''. {In what follows we frequently use results from the appendices.}
\begin{thm}[Standard Koszulity]\label{ThmKosz}
Let $\KKK$ be a finitely generated ideal in $(\fh^\ast,\le)$. The category ${}^{\KKK}\bO$ admits a graded cover ${}^{\KKK}\bO^{\mZ}$ such that simple, Verma and dual Verma modules admit graded lifts. We use the same symbol for the graded lifts and can choose the normalisation such that, for any $\mu\in\KKK$,
we have non-zero morphisms $\Delta(\mu)\to L(\mu)\to\nabla(\mu)$ (without applying shifts~$\langle k\rangle$) in ${}^{\KKK}\bO^{\mZ}$.
Then, for all $\mu,\nu\in\KKK$, we have
$$\Ext^i_{{}^{\KKK}\bO^{\mZ}}(\Delta(\mu),L(\nu)\langle j\rangle)\;=\;0\;=\; \Ext^i_{{}^{\KKK}\bO^{\mZ}}(L(\mu),\nabla(\nu)\langle j\rangle),\quad\mbox{if $i\not=j$}.$$
\end{thm}
\begin{proof}
It suffices to take an arbitrary $\lambda\in\KKK$ and {consider} ${}^{\KKK}\bO_{\lbra\lambda\rbra}$.
Set $A:=A^{\KKK}_{\lbra\lambda\rbra}$ and {recall} the equivalence
$$\Ffun:\;{}^{\KKK}\bO_{\lbra\lambda\rbra}\;\stackrel{\sim}{\to}\;A\mbox{-Mod}$$
from Theorem~\ref{ThmAMod}. For each $n\in\mN$ large enough we define the ideal $\mathring{\KKK}_n$ in $(\fh^\ast,\le)$ as in \ref{ForEquivK}.  We have the idempotents $\varepsilon_n\in A$ from Theorem~\ref{Thmidemptr}, {such that $A\cong\varinjlim_n A_n$ for $A_n=\varepsilon_n A\varepsilon_n$}. {According to} Proposition~\ref{PropOSK}, the algebras~$A_n$ have Koszul grading{s}. By Theorem~\ref{ThmADL}(ii), the grading on $A_n$ inherited from the one on  $A_{n+1}$ via the relation $\varepsilon_n A_{n+1}\varepsilon_n=A_n$ is also Koszul. By uniqueness of Koszul gradings, see e.g.~\cite[Corollary~2.5.2]{BGS}, the gradings on the algebras~$\{A_n\}$ are thus consistent and induce a grading on $A\cong\varinjlim_n A_n$. Example~\ref{ExAg} {shows that} the category
$${}^{\KKK}\bO_{\lbra\lambda\rbra}^{\mZ}:=A\mbox{-gMod}$$ is a graded cover of ${}^{\KKK}\bO_{\lbra\lambda\rbra}$.

{Next}, for $\mu\in\KKK$, we consider the $A$-module $M:=\Ffun(\Delta(\mu))$. By observing that $M=\varinjlim \varepsilon_nM$ and using the fact that the $A_n$-module $\varepsilon_nM$ is uniquely gradable up to shift, it follows that $M$ admits a graded lift.
 We thus have a projective resolution of $M$ in $A$-gMod. For $n$ large enough {so} that $\mu\in\KKK_n$, it follows from Lemma~\ref{CorNPM} (or Corollary~\ref{CornCoh}) that all terms in {this} complex are direct sums of modules $P_{\KKK}(\kappa)$ with $\kappa\in \KKK_n$.
The exact full functor
\begin{equation}\label{eqgMod}\varepsilon_n: A\mbox{-gMod}\to A_n\mbox{-gMod}\end{equation}
shows via the standard Koszulity of $A_n$ that $\Ext^i_{{}^{\KKK}\bO_{\lbra\lambda\rbra}^{\mZ}}(\Delta(\mu),L(\nu)\langle j\rangle)=0$, if $i\not=j$. The statement for dual Verma modules follows similarly.
\end{proof}

The following proposition suggests that any complete theory of Koszul {\em duality} for $\bO$ would restrict to a duality between dominant and antidominant blocks. For $\mu\in\KKK$, $j\in\mZ$ and
$M\in {}^{\KKK}\bO^{\mZ}$, we set
$$[M:L(\mu)\langle j\rangle]=\dim\Hom_{{}^{\KKK}\bO^{\mZ}}(P_{\KKK}(\mu)\langle j\rangle ,M),$$
with $P_{\KKK}(\mu)$ the projective cover of $L(\mu)\langle 0\rangle$ in ${}^{\KKK}\bO^{\mZ}$.
\begin{prop}
Let $\lambda,\mu\in\fh^\ast$ be integral and regular, with $\lambda$ dominant and $\mu$ antidominant. For all $w,x\in W$ and $j\in\mN$, we have
\begin{enumerate}[(i)]
\item  $\dim\Ext^j_{\bO}(\Delta(w\cdot\lambda),L(x\cdot\lambda))\;=\; [\Delta(w^{-1}\cdot\mu):L(x^{-1}\cdot\mu)\langle j\rangle]$,
\item $\dim\Ext^j_{\bO}(\Delta(w\cdot\mu),L(x\cdot\mu))\;=\; [\Delta(w^{-1}\cdot\lambda):L(x^{-1}\cdot\lambda\langle j\rangle)]$.
\end{enumerate}
\end{prop}
\begin{proof}
Take $n\in\mN$ big enough such that $w\cdot\lambda-x\cdot\lambda\in\mZ\rR_n$ and the corresponding conditions for $\mu$ and $x^{-1},w^{-1}$ are satisfied. By Theorem~\ref{ExtDL}, the left-hand sides equal the corresponding dimensions in $\cO(\fg_n,\fb_n)$. Choosing an appropriate finitely generated ideal $\KKK\subset\fh^\ast$ and using equation~\eqref{eqgMod} shows that the right-hand side can be computed in $\cO^{\mZ}(\fg_n,\fb_n)$. The result thus follows from \cite[Proposition~1.3.1]{BGS}.
\end{proof}

Despite the fact that the property in Theorem~\ref{ThmKosz} implies ordinary Koszulity in the case of finite dimensional (quasi-hereditary) algebras, see Theorem~\ref{ThmADL}(i), we still have the following open question.
\begin{que}
{Is} 
$\Ext^i_{{}^{\KKK}\bO_{\lbra\lambda\rbra}^{\mZ}}(L(\mu),L(\nu)\langle j\rangle)\;=\;0$ {for} $i\not=j$?
\end{que}
The difficulty in answering this question lies in the fact that the indecomposable projective modules appearing in a fixed position in the projective resolution of a simple module in ${}^{\KKK}\bO_{\lbra\lambda\rbra}$ will generally form a set $\{P(\mu)\,|\,\mu\in S\}$, for some multiset of weights $S$ which is not lower finite. This already happens for instance in the projective cover of the kernel of $P_{\KKK}(0)\tto L(0)$.

Another open question is whether we can construct a cover without taking a Serre subcategory of $\bO$ via truncation.
\begin{que}
Is it possible to construct a graded cover of ${\bf O}$?
\end{que}






\section{The semiregular bimodule}

\subsection{Definitions}

\subsubsection{The group $\Gamma$}\label{defGamma} Let $S$ be a countable set.
We consider the free abelian group $\Gamma_S\in\Ab$ with basis $S$,
$$\Gamma_S:=\bigoplus_{s\in S}\mZ \qquad\mbox{with group homomorphism}\quad \htt:\Gamma_S\to \mZ{,} \quad (a_s)_{s\in S}\mapsto\sum_{s\in S}a_s.$$
Hence $\Gamma_S$ is isomorphic either to~$\mZ^{\oplus k}$ for some $k\in\mN$, or to $\mZ^{\oplus\aleph_0}$. In the following we {omit} the reference to~$S$.

For any two $\Gamma$-graded vector spaces~$V=\bigoplus_a V^{a}$ and~$W=\bigoplus_a W^{a}$, we define the $\Gamma$-graded vector space $\fHom_{\mk}(V,W)$ by {setting}
$$\fHom_{\mk}(V,W)^a\;{:}=\;\{f\in \Hom_{\mk}(V,W)\;|\; \mbox{with }\; f(V^b)\subset W^{b+a}\;\mbox{for all $b\in\Gamma$}\}.$$
We equip the one dimensional vector space $\mk$ with the trivial $\Gamma$-grading. Then $\fHom_{\mk}(V,\mk)$ is the subspace of~$V^\ast$ of functionals which vanish {at} all but finitely many degrees. {We write  $V^{\circledast}=\fHom_{\mk}(V,\mk)$ and}  will interpret $(-)^\circledast$ as a duality functor on the category of~$\Gamma$-graded vector spaces with finite dimensional graded components.

\subsubsection{}We will work with $\Gamma$-graded Lie algebras over $\mk$, denoted by
$\fk=\bigoplus_{a\in\Gamma}\fk^a.$
Any $\Gamma$-graded Lie algebra has an {\bf associated $\mZ$-grading} through the homomorphism $\htt$:
$$\fk\;=\;\bigoplus_{i\in\mZ}\fk^{(i)},\qquad\fk^{(i)}=\bigoplus_{\htt(a)=i}\fk^{a}.$$

For $\Gamma$-graded $\fk$-modules $M,N$ we write $\fHom_{U(\fk)}(M,N)$ for the subspace of $\fHom_{\mk}(M,N)$ of $\fk$-linear morphisms.

\begin{ddef}\label{defg}
We say that a $\Gamma$-grading on a Lie algebra~$\fk$ is {\bf triangular} if
\begin{enumerate}[(i)]
\item $\fk^a=0$, whenever $a=(a_s)\in \Gamma$ contains both positive and negative integers;
\item $\dim_{\mk}\fk^a<\infty$ if $\htt(a)<0$;
\item $\fk$ is generated by the subspace $\fk^{(1)}\oplus\fk^{(0)}\oplus\fk^{(-1)}$.
\end{enumerate}
\end{ddef}
Condition~(i) implies in particular that~$\fk^{(0)}=\fk^{\fnul}$.
\begin{ex}
Definition~\ref{defg} is tailored to cover Kac-Moody algebras for arbitrary (possibly infinite) generalised Cartan matrices. The group $\Gamma$ is then to be identified with the root lattice. When the Cartan matrix is finite (and hence $\Gamma$ is finitely generated) the spaces~$\fk^{(i)}$ are already finite dimensional. In this case, one might as well work with the associated $\mZ$-grading.
\end{ex}



\subsubsection{}\label{defBN} For a triangularly~$\Gamma$-graded Lie algebra~$\fk$, we set
$$\fk_{<}{:}=\bigoplus_{i<0}\fk^{(i)},\;\;\fk_{\ge}{:}=\bigoplus_{i\ge 0}\fk^{(i)},\;\; N{:}=U(\fk_{<}),\;\; B{:}=U(\fk_{\ge}),\mbox{ and }\; U{:}=U(\fk).$$
All these algebras are naturally~$\Gamma$-graded.
\subsubsection{Semi-infinite characters}\label{DefSemi} Consider a triangularly~$\Gamma$-graded Lie algebra~$\fk$.
Following \cite[Definition~1.1]{Soergel}, see also \cite{Arkhipov}, we call a character $\gamma:\fk^{\fnul}\to\mk$ {\bf semi-infinite for~$\fk$} if
$$\gamma([X,Y])\,=\,\tr(\ad_X\ad_Y:\fk^{\fnul}\to\fk^{\fnul})\quad \mbox{for all $X\in\fk^{(1)}$ and~$Y\in\fk^{(-1)}$}.$$

\subsection{Some bimodules}
Keeping notation as above, we consider a triangularly~$\Gamma$-graded Lie algebra~$\fk$.

\subsubsection{The bimodule~$N^{\circledast}$}\label{Nast}We have the natural $N$-bimodule structure on $N^\ast=\Hom_{\mk}(N,\mk)$, with $(fn)(u)=f(nu)$ and~$(nf)(u)=f(un)$, for~$f\in N^\ast$ and~$u,n\in N$.
The subspace
$$N^{\circledast}:=\fHom_{\mk}(N,\mk)$$
clearly constitutes a sub-bimodule of~$N^\ast$.

\subsubsection{}
The $(N,B)$-bimodule structure on $N^{\circledast}\otimes_{\mk}B$ is induced from the left $N$-module structure on $N^{\circledast}$ and the right module structure on $B$.
The $N$-bimodule structure on $N^{\circledast}$ and {the} $(N,U)$-bimodule {structure on $U$} yield an $(N,U)$-bimodule structure on $N^{\circledast}\otimes_{N}U$.

\subsubsection{} Now fix an arbitrary character $\gamma:\fk^{(0)}\to\mk$ and define the one dimensional left $B$-module~$\mk_\gamma$ via the character $\gamma:\fk^{(0)}\to \mk$ and the surjection $\fk_{\ge}\tto\fk^{(0)}$.
Then we have the $B$-bimodule~$\mk_\gamma\otimes_{\mk}B$, which as a left module is the tensor product of~$\mk_\gamma$ and the left regular module. The {structure of right $B$-module comes from $B$ as a right $B$-module. Next, considering}
 $U$ as a $(B,U)$-bimodule, allows to introduce the $(U,B)$-bimodule
$\fHom_B(U,\mk_\gamma\otimes B)$.

\begin{lemma}\label{LemLR}We consider arbitrary elements~$n\in N$, $b,b'\in B$ and~$f\in N^{\circledast}$.
\begin{enumerate}[(i)]
\item The $(N,B)$-bimodule morphism
$$\psi:N^{\circledast}\otimes_{\mk}B\,\to\,N^{\circledast}\otimes_{N}U,\qquad\psi(f\otimes b)=f\otimes b $$ is an isomorphism. 
\item
The $(N,B)$-bimodule morphism
$$\phi:N^{\circledast}\otimes_{\mk}B\,\to\, \fHom_{B}(U,\mk_\gamma\otimes B),\quad\phi(f\otimes b)(b'n)=b'(f(n)\otimes b)$$ 
is an isomorphism. \end{enumerate}\end{lemma}
\begin{proof}
These are immediate applications of the PBW theorem. 
\end{proof}

\subsection{The semi-regular bimodule}
We continue with assumptions and notation as above and now also {\em assume that~$\gamma:\fk^{(0)}\to\mk$ is a semi-infinite character for~$\fk$.} On the space $N^{\circledast}\otimes_{\mk}B$, we can define a right $U$-action through the isomorphism $\psi$ in Lemma~\ref{LemLR}(i){,} and a left $U$-action through the isomorphism $\phi$ in Lemma~\ref{LemLR}(ii).

\begin{prop}\label{PropBM}
The left and right $U$-action on $N^{\circledast}\otimes_{\mk}B$ commute if $\gamma$ is a semi-infinite character. 
\end{prop}
\begin{proof}
This results from the same reasoning as in the proof of \cite[Theorem~1.3]{Soergel}. 
By construction, we only need to prove that the left $B$-action commutes with the right $N$-action.

For the left $B$-action it suffices to consider the action of~$\fk^{(0)}\oplus \fk^{(1)}$, by \ref{defg}(iii). For~$H\in \fk^{(0)}$, $f\in N^{\circledast}$ and~$b\in B$, a direct computation shows that
\begin{equation}\label{eqHact}H(\phi(f\otimes b))\;=\; -\phi(f\circ\ad_H\otimes b)+\gamma(H)\phi(f\otimes b)+\phi(f\otimes Hb).\end{equation}
Note that~$\ad_H\in \End_{\mk}(N)^{0} \subset\fEnd_{\mk}(N)$, so that~$f\circ\ad_H\in N^{\circledast}$ is well-defined. That this left action commutes with the right $N$-action follows as in \cite[Theorem~1.3]{Soergel}.
Now we consider the left action of~$\fk^{(1)}$. By~\ref{defg}(ii), $\fk^{(1)}$ is spanned by vectors $X\in \fk^{\gamma}$ for basis elements~$\gamma\in S\subset\Gamma$. For such $X$, by~\ref{defg}(i) we then find that the dimension of~$[X,\fk_{<0}]\cap\fk^0=[X,\fk^{-\gamma}]$ is finite. We take a basis $\{H_i\}$ of this space, which allows to define  $H^i, F\in\fEnd_{\mk}(N)$ by
$$nX\;=\; Xn+\sum_i H_iH^i(n)+F(n)\;\mbox{ in $U(\fk)$}\qquad\mbox{for all $n\in N$}.$$
A direct computation shows that
\begin{equation}\label{eqXact}X\phi(f\otimes b)\;=\;\phi(f\otimes Xb)+\phi(f\circ F\otimes b)+\sum_i\gamma(H_i)\phi(f\circ H^i\otimes b)+\sum_i\phi(f\circ H^i\otimes H_ib).\end{equation}
That this action commutes with the left $N$-action follows again from the same computation as in \cite[Theorem~1.3]{Soergel}.
\end{proof}

The resulting bimodule in Proposition~\ref{PropBM} will be denoted by~$S_{\gamma}$, and referred to as the {\bf semi-regular} bimodule. 

\begin{cor}\label{CorAdH}
Consider the inclusion of~$N$-bimodules~$\iota: N^{\circledast}\hookrightarrow S_\gamma$, corresponding to~$N^{\circledast}\hookrightarrow N^{\circledast}\otimes_{\mk}B=S_\gamma$.
\begin{enumerate}[(i)]
\item The adjoint action of~$H$ on the bimodule~$S_\gamma$ satisfies
$$\ad_H(\iota(f))\;=\;\gamma(H)\iota(f)-\iota(f\circ \ad_H)\quad\mbox{for~$H\in\fk^{(0)}$ and~$f\in N^{\circledast}$.}$$
\item The $(U,N)$-bimodule morphism 
$$\xi:U\otimes_N N^{\circledast}\;\to\;S_\gamma,\quad u\otimes f\mapsto u\iota(f),$$
is an isomorphism.
\end{enumerate}
\end{cor}
\begin{proof}
Part~(i) is essentially equation~\eqref{eqHact}.

For part (ii), we will prove that the composition $\sigma:=\phi^{-1}\circ\xi$ with $\phi$ from lemma \ref{LemLR}(ii)
$$\sigma:\;B\otimes_{\mk}N^{\circledast}\to N^{\circledast}\otimes_{\mk}B,\quad b\otimes f\mapsto \phi^{-1}(b\phi(f\otimes 1)),$$
is an isomorphism.
Consider arbitrary $X_1,\ldots, X_k\in \fk^{(0)}\cup\fk^{(1)}$. Equations~\eqref{eqHact} and~\eqref{eqXact} imply that, for $f\in N^\circledast$, we have
$$\xi(X_1\cdots X_k\otimes f)\;=\; f\otimes X_1\cdots X_k\;+\;\sum g\otimes u,$$
where $\sum g\otimes u$ stands for a finite sum of elements~$g\otimes u$, where $g\in N^{\circledast}$ and~$u\in B$ such that~$u$ is
\begin{itemize}
\item a product of strictly fewer than $k$ elements of $\fk^{(0)}\cup\fk^{(1)}$
\end{itemize}
{or}
\begin{itemize}
\item a product of exactly $k$ elements of $\fk^{(0)}\cup\fk^{(1)}$, but strictly more elements belonging to $\fk^{(0)}$ than in~$X_1\cdots X_k$.
\end{itemize}
From this, it is easy to show that~$\sigma$ must be an isomorphism.
\end{proof}


\section{Ringel duality}
Now we return to {a} root-reductive Lie algebra~$\fg$ as in the beginning of Section~\ref{Sec4}.
\subsection{Triangular $\Gamma$-grading and semi-infinite characters}

\subsubsection{}
Using the notation of~\ref{defGamma}, we set 
$$\Gamma:=\Gamma_{\Sigma }\cong \mZ\Sigma \cong\mZ\rR.$$
The root decomposition~\eqref{rootdec}, where $\fh=\fg^0$, is thus a $\Gamma$-grading.
It is easily checked that this makes $\fg$ a triangularly~$\Gamma$-graded Lie algebra. We then have $\fg_{\ge}=\fb$ and~$\fg_{<}=\fn^-$, and thus $B=U(\fb)$ and~$N=U(\fn^-)$, for the algebras introduced in \ref{defBN}.

\subsubsection{}By the above, we can view $\Gamma$ as a subgroup of~$\fh^\ast$ {and write} $\sigma:\Gamma\hookrightarrow \fh^\ast$. In particular, this equips any $\Gamma$-graded vector space $V$ with the structure of a semisimple $\fh$-module, by setting $H(v)=\sigma(\gamma)(H)v$, for any $v\in V_\gamma$ and~$H\in\fh$.
The dual $V^{\circledast}$ of~\ref{defGamma} then corresponds to the finite dual of~$V$ as a semisimple $\fh$-module{, see}~\ref{SecWeightM}.
In particular, we can interpret $N^\circledast$ as in \ref{Nast} in this way by using the adjoint $\fh$-action.


\begin{lemma}\label{Lem2rho}The semi-infinite characters $\gamma\in\fh^\ast$ are those characters $\gamma:\fh\to\mk$, for which~$\gamma(H)=2\rho(H)$ for all $H\in \fh\cap[\fg,\fg]${, for $\rho$ as defined} in \ref{rhoshift}.
\end{lemma}
\begin{proof}
For each simple positive root $\alpha$, we consider the Chevalley generator{s}~$E_\alpha$ and~$F_\alpha$, and set $ H_\alpha:=[E_\alpha,F_\alpha]$. By applying the definition in \ref{DefSemi}, we find
$$\gamma([E_\alpha,F_\alpha])=\alpha(H_\alpha)={2\rho(H_\alpha)}.$$
 The conclusion then follows, since $H\in \fh\cap[\fg,\fg]$ is spanned by vectors $H_\alpha$ as above, for a Dynkin Borel subalgebra.
\end{proof}
This determines all semi-infinite characters for Dynkin Borel subalgebras in case $\fg$ is simple.
\begin{cor}
For~$\fg$ equal to~$\mathfrak{sl}_\infty$, $\mathfrak{so}_\infty$ or~$\mathfrak{sp}_\infty$, the unique semi-infinite character is $2\rho$.
\end{cor}


\subsection{The AS duality functor} In this subsection, we consider the analogue of the duality functor constructed by Arkhipov and Soergel for (affine) Kac-Moody algebras in \cite{Arkhipov, Soergel}.

We set $\gamma=2\rho$, which is a semi-infinite character by Lemma~\ref{Lem2rho}, and consider the corresponding semiregular bimodule   {$S:=S_{2\rho}$}.


\begin{lemma}\label{TwistVerma}
For any $\lambda\in\fh^\ast$, we have an isomorphism $S\otimes_U\Delta(\lambda)\,\cong\,\Delta(-\lambda-2\rho)^{\circledast}$.
\end{lemma}
\begin{proof}
Using the notation of Corollary~\ref{CorAdH} we find that~$S\otimes_U\Delta(\lambda)$ is equal to its subspace
$\iota(N^{\circledast})\otimes \mk_\lambda,$
with $\mk_\lambda$ the one dimensional subspace of~$\Delta(\lambda)$ of weight~$\lambda$. By Corollary~\ref{CorAdH}(i) we have for any $H\in\fh$, $f\in N^\circledast$ and~$v\in \mk_\lambda$
$$H(\iota(f)\otimes v)\;=\; 2\rho(H)(\iota(f\otimes v))+\iota(f)\otimes Hv-\iota(f\circ\ad_H)\otimes v.$$
Hence, {there is} an isomorphism
\begin{equation}\label{isoresh}\res^{\fg}_{\fh}S\otimes_U\Delta(\lambda)\;\cong\; N^{\circledast}\otimes_{\mk}\mk_{\lambda+2\rho},\end{equation}
for the canonical adjoint $\fh$-action on $N^{\circledast}$.
In particular, $S\otimes_U\Delta(\lambda)$ is a weight module, so Lemma~\ref{stupidlemma} implies
$$(S\otimes_U\Delta(\lambda))^{\circledast}\cong \Delta(\mu)$$ for some $\mu\in\fh^\ast$. Equation~\eqref{isoresh} implies that $\mu=-\lambda-2\rho$.
\end{proof}

\begin{lemma}\label{LemFF}
The functor~$\FF:\cF^\Delta(\fg,\fb)\to \cF^{\nabla}(\fg,\fb^-)$, obtained by the restriction of~$S\otimes_U-$, is an equivalence of exact categories.
\end{lemma}
\begin{proof}
We start by considering the functor~$S\otimes_U-\,:\cF^\Delta(\fg,\fb)\to U$-Mod. Since 
$$\res^{\fg}_{\fn^-}S\otimes_U-\;\cong \;N^{\circledast}\otimes_N\res^{\fg}_{\fn^-}-,$$ Lemma~\ref{stupidlemma} implies that~$S\otimes_U-$ is exact. Lemma~\ref{TwistVerma} then implies that the image of objects in~$\cF^\Delta(\fg,\fb)$ are contained in~$ \cF^{\nabla}(\fg,\fb^-)$. We denote the corresponding exact functor by~$\FF$.

By tensor-hom adjunction, we have the right adjoint functor~$\Hom_U(S,-)$. By Corollary~\ref{CorAdH}(ii), we have an isomorphism of functors
$$\res^{\fg}_{\fn^-}\circ\Hom_{U}(S,-)\;\cong\; \Hom_{N}(N^{\circledast},\res^{\fg}_{\fn^-}-).$$
Hence, this yields an exact functor~$\GG: \cF^{\nabla}(\fg,\fb^-)\to \cF^\Delta(\fg,\fb)$. That the adjoint {functors} $(\FF,\GG)$ are mutually inverse follows as in \cite[Theorem~2.1]{Soergel}.
\end{proof}

We can compose the functor~$\FF$ with the duality functor~$(-)^\circledast$ on $\cC(\fg,\fh)$, {and} denote the corresponding functor by~$\DD$. 

\begin{cor}${}$\label{PropRD}
The functor~$\DD$ yields an exact contravariant equivalence ${\cF}^{\Delta}\;\tilde\to\;{\cF}^{\Delta}$, mapping~$\Delta(\lambda)$ to~$\Delta(-\lambda-2\rho)$.
\end{cor}
\begin{proof}
This is immediate from Lemmata~\ref{LemFF}, \ref{TwistVerma} and Corollary~\ref{CorDua}.
\end{proof}

\subsection{Ringel duality and tilting modules}

We can also compose the functor~$\FF$ with the twist by the automorphism $\tau$, or equivalently, the functor~$\DD$ with the duality functor $\vee$ on $\cC(\fg,\fh)$ of~\ref{SecWeightM}. By comparing the following proposition with Theorem~\ref{ThmRingelAlg}(iii), we can interpret 
$$\RR:={}_{\tau}S\otimes_U\cong (-)^\vee\circ\DD$$ as the Ringel duality functor.
\begin{prop}[Ringel self-duality of $\bO$] \label{PropRD}
 The functor~$\RR$ yields an exact equivalence ${\cF}^{\Delta}\;\tilde\to\;{\cF}^{\nabla}$, mapping $\Delta(\lambda)$ to~$\nabla(-\lambda-2\rho)$.
\end{prop}
\begin{proof}
This is immediate from Lemmata~\ref{LemFF}, \ref{TwistVerma} and Corollary~\ref{CorDua}.
\end{proof}

\begin{rem}
By Theorem~\ref{ThmRingelAlg}(iii), we can thus state that~$\bO_{\lbra\lambda\rbra}$ is {\em Ringel dual} to~$\bO_{\lbra-\lambda-2\rho\rbra}$. \end{rem}

The following Proposition represents the combinatorial shadow of the Ringel duality between $\bO_{\lbra\lambda\rbra}$ and~$\bO_{\lbra-\lambda-2\rho\rbra}$, see Theorem~\ref{ThmRingelAlg}(iv).
\begin{prop} \label{PropRing2}
Let $\CCC\subset\fh^\ast$ be a lower finite coideal and~$\nu\in\CCC$. There exists a module $T_{\CCC}(\nu)\in \cF^{\nabla}[\CCC]$ such that, for all $\kappa\in\CCC$,
$$(T_{\CCC}(\nu):\nabla(\kappa))\;=\;[\Delta(-\kappa-2\rho):L(-\nu-2\rho)]\qquad\mbox{and}\qquad \Ext^1_{\bO}(T_{\CCC}(\nu),\nabla(\kappa))\;=\;0.$$
\end{prop}
\begin{proof}
We define the upper finite ideal
$$\KKK\;:=\;\{-\mu-2\rho\,|\, \mu\in\CCC\}$$
and the module $N:=\RR(P_{\KKK}(\lambda))$, with $\lambda:=-\nu-2\rho$. We use freely the results of Theorem~\ref{ThmProj}. By Proposition~\ref{PropRD}, we have $N\in \cF^{\nabla}[\CCC]$ with
$$(N:\nabla(\kappa))\;=\; (P_{\KKK}(\lambda):\Delta(-\kappa-2\rho))\;=\;[\Delta(-\kappa-2\rho):L(\lambda)].$$
By Proposition~\ref{PropRD}, we also have
$$\Ext^1_{\bO}(N,\nabla(\mu))\;=\;\Ext^1_{\bO}(P_{\KKK}(\lambda),\Delta(-\mu-2\rho))\;=\;0\quad\mbox{for all $\mu\in\CCC$}.$$
This concludes the proof.\end{proof}

\begin{ex}
For $\nu\in\fh^\ast$ we set $\CCC:=\{\lambda\in\fh^\ast\,|\, \lambda\ge\nu\}$. Then  $T_{\CCC}(\nu)=\nabla(\nu)$.
\end{ex}

\begin{rem}
The vanishing of extensions in Proposition~\ref{PropRing2} implies that inside the quotient ${}_{\LLL}\bO$, for $\LLL=\fh^\ast\backslash\CCC$, the module $T_{\CCC}(\nu)$ becomes a tilting module, that is a module with Verma and dual Verma flag. This follows from \cite[Theorem~4]{Ringel} and the results in Section~\ref{Sec4}.
\end{rem}

As a special case of the above remark, we have the following corollary, which also follows from Corollary~\ref{CorEqBlocks}(ii). 

\begin{cor}
If $\lambda\in\fh^\ast$ is antidominant, we have a $\fg$-module $T(\mu)$ for each $\nu\in \lbra\lambda\rbra$ which is in~$\cF^{\Delta}\cap\cF^{\nabla}$ and satisfies
$$(T(\nu):\nabla(\kappa))\;=\;[\Delta(-\kappa-2\rho):L(-\nu-2\rho)]\quad\mbox{for all $\kappa\in \lbra\lambda\rbra$.}$$

\end{cor}


\appendix

\section{Homological algebra in Serre quotient categories}

\label{AppSerre}

\subsection{Serre quotient categories}We recall some results from \cite[Chapitre~III]{Gabriel}. We fix an abelian category $\cC$ with Serre subcategory~$\cB\subset\cC$ for the entire subsection.

\subsubsection{}The Serre quotient category~$\cC/\cB$ is defined by setting~$\Ob(\cC/\cB){:}=\Ob \cC$ and for $X,Y\in\cC$
$$\Hom_{\cC/\cB}(X,Y)\;:=\;\varinjlim \Hom_{\cC}(X',Y/Y'),$$
where~$X'$, resp. $Y'$, runs over all subobjects in~$\cC$ (ordered by inclusion) of~$X$, resp. $Y$, such that~$X/X'\in\cB\ni Y'$. For the precise definition of the composition of two morphisms in~$\cC/\cB$ we refer to \cite{Gabriel}.

By \cite[Proposition~III.1.1]{Gabriel}, the category~$\cC/\cB$ is abelian and we have an exact functor~$\pi:\cC\to\cC/\cB$, which is the identity on objects and is given on morphisms by the canonical morphism from
$\Hom_{\cC}(X,Y)$ to $\varinjlim \Hom_{\cC}(X',Y/Y').$


The following is the universality property of Serre quotient categories.
\begin{lemma}\cite[Corollaires~III.1.2 and III.1.3]{Gabriel}\label{factorspi}
Assume that for an abelian category $\cC'$ and an exact functor~$\Ffun :\cC\to\cC'$, we have $\Ffun( X)=0$ for all $X\in\cB$. {Then} there exists a unique functor~$\widetilde{\Ffun }$ {which makes the diagram }
$$\xymatrix{
\cC\ar[rr]^{\Ffun}\ar[d]^{\pi}&&\cC'\\
\cC/\cB\ar[rru]^{\widetilde{{\Ffun}}}
}$$
{commutative.}
Furthermore, the functor $\widetilde{{\Ffun}}$ is exact
and  $\widetilde{{\Ffun}}(f)={\Ffun}(g)$ for any $f\in \Hom_{\cC/\cB}(X,Y)$ represented by some $g\in \Hom_{\cC}(X',Y/Y')$.
\end{lemma}

\begin{lemma}\label{ProjBC}
Consider a projective object $P$ in~$\cC$ with a unique maximal subobject $X$, and assume $P/X$ is not an object of~$\cB$.
 Then $P$ is also projective in~$\cC/\cB$, and {there is an isomorphism of functors}
$$\Hom_{\cC/\cB}(\pi P,\pi -)\;=\;\Hom_{\cC}(P,-).$$
\end{lemma}
\begin{proof}
By assumption, there is no subobject~$P'\subset P$ with $P/P'\in\cB$. This implies
$$\Hom_{\cC/\cB}(P,M)\;=\;\varinjlim\Hom_{\cC}(P,M/M')\quad \mbox{for all $M\in\cC$,}$$
where the limit is taken over all $\cB\ni M'\subset M$. By assumption, $\Hom_{\cC}(P,M')=0$, so all morphisms in the limit are isomorphisms
and the diagram
$$\xymatrix{
\cC\ar[rrrr]^{\Hom_{\cC}(P,-)}\ar[rrd]^{\pi}&&&&\Ab\\
&&\cC/\cB\ar[rru]_{\Hom_{\cC/\cB}(P,-)}
}$$
commutes. {The exactness of} $\Hom_{\cC/\cB}(P,-)$ thus follows from Lemma~\ref{factorspi}.
\end{proof}

\subsection{Example: Locally unital algebras}

Let $A$ be a locally unital algebra with mutually orthogonal idempotents $\{e_\alpha\,|\,\alpha\in\Lambda\}$. For any subset $\Lambda'\subset \Lambda$, we have the locally unital algebra 
$$A'\;=\;\bigoplus_{\alpha,\beta\in \Lambda'}e_\alpha A e_\beta.$$

\begin{lemma}\label{LemQuoAlg}
Set {\rm $\cC=A$-Mod} and $\cC'=A'${\rm -Mod}.
With
$\cB$ the Serre subcategory of $A$-modules which satisfy $e_\alpha M=0$ for all $\alpha\in\Lambda'$, we have $\cC/\cB\cong \cC'$.
\end{lemma}
\begin{proof}
Consider the exact functor
$${\Ffun}:\cC\to \cC',\quad M\mapsto \bigoplus_{\alpha\in\Lambda'}e_\alpha M$$
and the right exact functor
$${\Kfun}=X\otimes_{A'}-:\cC'\to\cC\quad\mbox{with $X=\bigoplus_{\alpha\in\Lambda'}Ae_\alpha$}. $$
Clearly, the composition ${\Ffun}\circ {\Kfun}$ is isomorphic the identity on $\cC'$, so $\Ffun$ is dense and full. It then follows from Lemma~\ref{factorspi} that we have a dense and full functor $\widetilde{{\Ffun}}:\cC/\cB\to \cC'$. 

Now we prove that $\widetilde{{\Ffun}}$ is faithful. Assume that $\widetilde{{\Ffun}}(f)=0$ for $f\in\Hom_{\cC/\cB}(M,N)$. We take a representative $g\in \Hom_{{\mathcal{C}}}(M',N/N')$ of $f$ with $M/M'\in\cB\ni N'$. Lemma~\ref{factorspi} implies ${\Ffun}(g)=0$. This means that the restriction of $g$ to $\bigoplus_{\alpha\in\Lambda'}e_\alpha M'$ is zero. It follows that $\im g\in\cB$. So we can define $N''\in\cB$ with $N'\subset N''\subset N$ and $N''/N'=\im g$. Hence, {$g$ is the mapped to zero by the map} $\Hom_{\cC}(M',N/N')\to\Hom_{\cC}(M',N/N'')$, which shows that $f=0$.
\end{proof}

\section{Ringel duality}

\subsection{Quasi-hereditary algebras}\label{QHAlg}
For a {\em finite dimensional, unital and basic} algebra~$A$, fix an orthogonal decomposition of the identity element $1=e_1+e_2+\cdots+e_n$ into primitive idempotents. We write $(A,\mathsf{e})$ when we consider the algebra $A$ together with the above {\em ordered} choice of primitive idempotents. Set 
$$\varepsilon_i{:}=e_i+\cdots+e_n\quad\mbox{for all $1\le i\le n,\;\,$ and }\quad\;\varepsilon_{n+1}{:}=0.$$

\subsubsection{}The {\bf standard modules} are given by
$$\Delta(i)\;=\;Ae_i/(A\varepsilon_{i+1}Ae_i),\qquad1\le i\le n.$$
We also have the projective cover $P(i)=Ae_i$ of the simple module~$L(i)$. 
{Then} $\Delta(n)=P(n)$ and, if $A$ has finite global dimension, also $\Delta(1)=L(1)$.

We denote the category of finite dimensional modules with $\Delta$-flag by~$\cF^{\Delta}_A$.
Dually {we} define the costandard modules~$\nabla(i)$ and the category~$\cF^{\nabla}_A$.

The algebra $(A,\mathsf{e})$ is {\bf quasi-hereditary}, see~\cite{CPS}, if and only if  $[\Delta(i):L(i)]=1$ for all $1\le i\le n$ and
 $A\in \cF_A^\Delta$. The following well-known lemma follows from the definition.

\begin{lemma}\label{LemBC}
If $(A,\mathsf{e})$ is quasi-hereditary, then both $\varepsilon_iA\varepsilon_i$ and~$A/(A\varepsilon_{i+1}A)$ are quasi-hereditary for any $1\le i\le n$. The order on the simple modules is induced from the one for $A$.
\end{lemma}

\subsection{Ringel duality}

The following is proved in \cite[Section~6]{Ringel} and \cite[Theorem~3]{DR}.
\begin{thm}\label{ThmRingelAlg}
Let $(A,\mathsf{e})$ denote a quasi-hereditary algebra. 
\begin{enumerate}[(i)]
\item The category $\cF_A^{\Delta}\cap\cF_A^{\nabla}$ contains precisely $n$ indecomposable non-isomorphic modules~$\{T(i)\,|\,1\le i\le n\}$. For every $1\le  i\le n$, there exists a short exact sequence
$$0\to N(i)\to T(i)\to \nabla(i)\to 0\quad\mbox{with $N(i)\in \cF^{\nabla}_A$.}$$
\item The {\bf Ringel dual} algebra $(A',\mathsf{f})$ of~$(A,\mathsf{e})$, defined by~$A':=\End_A(\oplus_i T(i))^{\op}$ and~$f_i:=1_{T(n+1-i)}$, is quasi-hereditary and satisfies $A''\cong A$.
\item The algebra $(A',\mathsf{f})$ is the unique (basic) quasi-hereditary algebra for which there exists an exact equivalence
$$\RR:\; \cF_A^{\nabla}\;\;\tilde{\to}\;\; \cF_{A'}^{\Delta}.$$
\item For all $1\le i,j\le n$, we have
$$(T_A(i):\nabla_A(j))\;=\;(P_{A'}(n+1-i):\Delta_{A'}(n+1-j)).$$
\end{enumerate}
\end{thm}

Note that we can take $\Hom_A(\oplus_i T_A(i),-)$ for the functor $\RR$ in (iii), yielding in particular $\RR(T_A(i))\cong P_{A'}(n+1-i)$.

\section{Graded covers}
We introduce `graded covers' of abelian categories $\cC$, similarly to \cite[Section~4.3]{BGS}.

\subsection{Definition}
By an {\bf abelian $\mZ$-category $\cG$}, we mean an abelian category with a strict $\mZ$-action. A strict $\mZ$-action is a collection of exact functors $\{\langle j\rangle\,|\, j\in\mZ\}$ on $\cG$, which satisfy $\langle i\rangle\langle j\rangle=\langle i+j\rangle$ and~$\langle0\rangle={\mathbf I}{\mathbf d}_{\cG}$.

\begin{ddef}\label{DefCover}
A {\bf graded cover} of~$\cC$ is an abelian $\mZ$-category $\cC^{\mZ}$ with an exact functor~$\GG:\cC^{\mZ}\to\cC$, such that 
\begin{enumerate}[(i)]
\item $\GG\langle i\rangle=\GG$ for all $i\in\mZ$;
\item for all $M,N\in \cC^{\mZ}$, the functor $\GG$ induces group isomorphisms 
$$\bigoplus_{i\in\mZ}\Ext^l_{\cC^{\mZ}}(M,N\langle i\rangle)\;\stackrel{\sim}{\to}\;\Ext^l_{\cC}(\GG M,\GG N)\quad\mbox{for all $l\in\mN$;}$$
\item all simple objects in~$\cC$ are isomorphic to objects in the image of~$\GG$.
\end{enumerate}
\end{ddef}

\subsection{Positively graded algebras}
\subsubsection{}We say that a locally unital algebra $A$ is $\mZ$-graded if $A=\bigoplus_{j\in\mZ}A_j$ with $A_jA_k=A_{j+k}$ and $e_\alpha\in A_0$ for all $\alpha$. The category of $\mZ$-graded locally unital $A$-modules with morphism preserving the grading is denoted by $A$-gMod. If $A$ is locally finite, we denote by $A$-gmod the full subcategory of $A$-gMod of locally finite dimensional modules. 

For $M\in A$-gMod and~$i\in\mZ$, the shifted module~$M\langle i\rangle$ is identical to~$M$ as an ungraded module, but with grading
$$M\langle i\rangle_j\;=\; M_{j-i}\quad\mbox{for all $j\in\mZ$}.$$ We denote the exact functor {which forgets} the grading by 
$$\GG: A\mbox{-gMod}\to A\mbox{-Mod}.$$
A module $\widetilde{M}\in A$-gMod satisfying $\GG \widetilde{M}\cong M$ is a {\bf graded lift} of $M$.

A $\mZ$-graded algebra $A$ is {\bf positively graded} if $A_i=0$ for $i<0$ and  $A_0$ is semisimple. Clearly, for such an algebra, the simple modules admit graded lifts which are contained in {any chosen} degree.

\begin{ex}\label{ExAg}
Let $A$ be a locally unital $\mZ$-graded algebra such that every simple module has a graded lift, for instance $A$ is positively graded. Then $A$-gMod is a graded cover of $A$-Mod for the forgetful functor {$\GG$}.
\end{ex}

\section{Standard Koszul algebras}\label{AppKos}
We review some results about Koszul quasi-hereditary algebras, based on \cite{ADL, BGS}. The algebra $A$ is assumed to be associative, unital, finite dimensional and basic. 

\subsection{Algebras}


\subsubsection{} 
Assume that $A$ is $\mZ$-graded.
The homomorphism spaces in~$A$-gmod are denoted by~$\hom_A${,} and the extension functors by~$\ext^k_A$. 
We take the convention that, unless otherwise specified, a graded lift of a simple, standard or projective module is normalised (using $\langle j\rangle$) such that the top is in degree zero. Similarly, graded lifts of costandard {or} injective modules are chosen to have their socle in degree zero.


\subsubsection{} Now assume that $A$ is positively graded. The subcategory of~$A$-gmod of modules~$M$ which satisfy
$$\ext^j_A(M,L\langle i\rangle)=0\qquad\mbox{if $i\not=j$},$$
for all simple modules~$L$ (contained in degree $0$), is denoted by~$\PL_A$. Similarly, the subcategory of~$A$-gmod of modules~$M$ which satisfy
$$\ext^j_A(L,M\langle i\rangle)=0\qquad\mbox{if $i\not=j$},$$
for all simple modules~$L$, is denoted by~$\IL_A$.
 The positively graded algebra~$A$ is {\bf Koszul} if all simple modules are in~$\PL_A$, or equivalently in $\IL_A$, see~\cite[Proposition~2.1.3]{BGS}.

\subsubsection{} 
Consider a quasi-hereditary algebra $(A,\mathsf{e})$ as in Section~\ref{QHAlg}, which is positively graded. Then $A$ is {\bf standard Koszul} if
 each standard module~$\Delta(i)$ is {an object of}~$\PL_A$ and each costandard module is {an object of}~$\IL_A$.
Note that, by construction, $\Delta(i)$ and $\nabla(i)$ admit graded lifts. By definition, standard Koszul algebras are thus assumed to be positively graded quasi-hereditary algebras.

If $A$ is positively graded, then so are $A/(AeA)$ and~$eAe$ for any idempotent $e\in A_0$.

\begin{thm}\cite[Theorem~1.4 and Proposition~3.9]{ADL}\label{ThmADL}
Let $(A,\mathsf{e})$ be a standard Koszul algebra.
\begin{enumerate}[(i)]
\item The positively graded algebra $A$ is Koszul.
\item The algebras~$A/(A\varepsilon_{i+1}A)$ and~$\varepsilon_i A\varepsilon_i$, from Lemma~\ref{LemBC}, are standard Koszul for all $1\le i\le n$. 
\end{enumerate}
\end{thm}

\subsection{Category $\cO$}

\begin{prop}\cite{ADL}\label{PropOSK}
Let $\fg$ be a (finite dimensional) reductive Lie algebra, with $\fh\subset\fb\subset\fg$ a Cartan subalgebra $\fh$, Borel subalgebra $\fb$ and a weight $\lambda\in\fh^\ast$. 
\begin{enumerate}[(i)]
\item The basic algebra $A$ for which
$A\mbox{{\rm -mod}}\,\cong\,\cO_{\lbra\lambda\rbra},$
{is standard Koszul for an appropriate positive grading}.
\item For $\KKK$ an ideal in $(\fh^\ast,\le)$, the basic algebra $A^{\KKK}$ for which
$A^{\KKK}\mbox{{\rm -mod}}\,\cong\,{}^{\KKK}\cO_{\lbra\lambda\rbra},$  is of the form $eAe$ for some idempotent $e\in A_0$, {and moreover $A^K$} is standard Koszul for the grading inherited from~$A$.
\end{enumerate}
\end{prop}
\begin{proof}
Part (i) is \cite[Corollary~3.8]{ADL}, which is based on results in \cite{BGS}. Part (ii) is an application of Theorem~\ref{ThmADL}(ii). Note that for this we should complete $\le$ on $W\cdot\lambda$ to an arbitrary total order such that $\KKK\cap W\cdot\lambda$ is still an ideal.
\end{proof}

\begin{flushleft}
	K. Coulembier\qquad \url{kevin.coulembier@sydney.edu.au}
	
	School of Mathematics and Statistics, University of Sydney, NSW 2006, Australia
	
	\vspace{2mm}
	
	I. Penkov\qquad \url{i.penkov@jacobs-university.de}

Jacobs University Bremen, 28759 Bremen, Germany
\end{flushleft}

\end{document}